 \documentclass[final,3p,times]{elsarticle}

\usepackage{amsmath,amssymb,amsfonts}
\usepackage{amsthm}
\usepackage[dvipsnames]{xcolor}
\usepackage{dsfont}

\usepackage{caption}
\usepackage{placeins}
\usepackage{rotating}
\usepackage{lscape}
\usepackage{ulem}

\usepackage{tikz}
\definecolor{orange}{HTML}{d9944f}
\usetikzlibrary{patterns,arrows,positioning,fit}
\usetikzlibrary{decorations.pathreplacing}
\usetikzlibrary{decorations.pathmorphing}

\usepackage[linesnumbered, ruled, vlined]{algorithm2e}
\usepackage{graphicx,subcaption,lipsum}
\usepackage{url}
\usepackage{doi}
\usepackage{hyperref}
\usepackage[capitalize]{cleveref}
\usepackage{subcaption}
\usepackage[labelfont=bf]{caption}
\usepackage{float}
\usepackage{orcidlink}
\usepackage{enumitem}

\renewcommand{\d}[1]{\textnormal{d}#1}

\newtheorem{theorem}{Theorem}[section]
\newtheorem{lemma}[theorem]{Lemma}
\newtheorem{remark}[theorem]{Remark}
\newtheorem{theorem-appendix}{Theorem A.\ignorespaces}
\newtheorem{lemma-appendix}[theorem-appendix]{Lemma A.\ignorespaces}
\newtheorem{remark-appendix}[theorem-appendix]{Remark A.\ignorespaces}

\usepackage{comment}

\newcommand{\AW}[1]{{\color{blue}#1}}

\journal{Applied Mathematics and Computation}
\usepackage{natbib}

\begin{document}

\begin{frontmatter}



\title{A nonstandard numerical scheme for a novel SECIR integro-differential equation-based model allowing nonexponentially distributed stay times}


\author[label1]{Anna Wendler\hspace{0.5mm}\orcidlink{0000-0002-1816-8907}}
\ead{anna.wendler@dlr.de}
\author[label1]{Lena Pl\"otzke\hspace{0.5mm}\orcidlink{0000-0003-0440-1429}}
\author[label1]{Hannah Tritzschak}
\author[label1,label2]{Martin J. Kühn\hspace{0.5mm}\orcidlink{0000-0002-0906-6984}}
\ead{martin.kuehn@dlr.de}

\affiliation[label1]{organization={Institute of Software Technology, Department of High-Performance Computing, German Aerospace Center},
            city={Cologne},
            country={Germany}}

\affiliation[label2]{organization={Bonn Center for Mathematical Life Sciences and Life and Medical Sciences Institute, University of Bonn, Bonn, Germany},
            city={Bonn},
            country={Germany}}

\begin{abstract}
Ordinary differential equations (ODE) are a popular tool to model the spread of infectious diseases, yet they implicitly assume an exponential distribution to describe the flow from one infection state to another. However, scientific experience yields more plausible distributions where the likelihood of disease progression or recovery changes accordingly with the duration spent in a particular state of the disease. Furthermore, transmission dynamics depend heavily on the infectiousness of individuals. The corresponding nonlinear variation with the time individuals have already spent in an infectious state requires more realistic models. The previously mentioned items are particularly crucial when modeling dynamics at change points such as the implementation of nonpharmaceutical interventions. In order to capture these aspects and to enhance the accuracy of simulations, integro-differential equations (IDE) can be used.  

In this paper, we propose a generalized model based on integro-differential equations with eight infection states. The model allows for variable stay time distributions and generalizes the concept of ODE-based models as well as IDE-based age-of-infection models. In this, we include particular infection states for severe and critical cases to allow for surveillance of the clinical sector, avoiding bottlenecks and overloads in critical epidemic situations.

On the other hand, a drawback of IDE-based models \AW{is that efficient numerical solvers are not as widely available as for ODE systems} and tailored schemes might be needed. We will extend a recently introduced nonstandard numerical scheme to solve a simpler IDE-based model. This scheme is adapted to our more advanced model and we prove important mathematical and biological properties for the numerical solution. Furthermore, we validate our approach numerically by demonstrating the convergence rate. Eventually, we also show that our novel model is intrinsically capable of better assessing disease dynamics upon the introduction of nonpharmaceutical interventions.


\end{abstract}


\begin{highlights}
\item A SECIR-type integro-differential equation (IDE) based model allowing for nonexponential stay times
\item Developing a nonstandard numerical solution scheme for IDE based models
\item Proving mathematical-biological properties for the numerical solution scheme and demonstrating convergence rate
\item Demonstrating the impact of realistic distributions for nonpharmaceutical interventions
\end{highlights}


\end{frontmatter}

\section{Introduction}
As recently demonstrated by SARS-CoV-2, outbreaks of infectious diseases can put humankind and human societies at immense difficulties through highly impacting personal rights and individual health, as well as state economies and public health. 

Mathematical models are an invaluable tool for predicting the spread of infectious diseases and the impact of nonpharmaceutical interventions. They provide the basis for developing appropriate mitigation strategies. There are a variety of approaches that can be employed. Models based on ordinary differential equations, or more generally ODE-based metapopulation models, are widely used~\cite{giordano_modelling_2020,bauer_relaxing_2021,schuler_data_2021,koslow_appropriate_2022,contento_integrative_2023} and can integrate many things such as mobility and waning immunity, see, e.g.,~\cite{zunker_novel_2024} or reinfections by novel variants, see, e.g.,~\cite{merkt_long-term_2024}. On a finer scale, agent-based models (ABMs) can be used to model individual courses of disease, see, e.g.,~\cite{bershteyn_implementation_2018,kerr_covasim_2021,muller_predicting_2021,bracher_pre-registered_2021,KKN24}. However, ABMs come with high-computational costs. To overcome these issues, several hybrid approaches have been proposed recently, see, e.g.,~\cite{bicker_hybrid_2025,maier_hybrid_2024}.

 To support decision makers in mitigating infectious disease spread, mathematical models should furthermore be integrated into automatic pipelines to provide permanently updated data~\cite{memonAutomatedProcessingPipelines2024} that are embedded in a visual analytics toolkit; cf.~\cite{betzESIDExploringDesign2023a}. 

Models based on ordinary differential equations are especially popular as they are mathematically well understood and suitable numerical solvers are widely available. However, these models implicitly assume an exponential stay time distribution in the compartments which was found unrealistic from an epidemiological point of view~\cite{wearing_appropriate_2005,donofrio_mixed_2004, faesTimeSymptomOnset2020, challenMetaanalysisSevereAcute2022, madewellSerialIntervalIncubation2023}. The resulting implications can be, e.g., underestimated basic reproduction numbers~\cite{wearing_appropriate_2005} that eventually lead to insufficient mitigation action.

The presented need for more general models that can naturally consider arbitrary stay time distributions leads to integro-differential equations, in particular form also sometimes denoted age-of-infection models~\cite{brauer_age_2009}. Already in $1927$, Kermack and McKendrick~\cite{kermack_contribution_1927} presented a general model using integro-differential equations. However, their paper is mostly cited for their simple ODE model~\cite{diekmann_legacy_1995}. Recently, different authors formulated age of infection models, see, e.g.,~\cite{messina_non-standard_2022, brauer_mathematical_2019, baiAgeofinfectionModelBoth2023, barrilReproductionNumberAge2021, demongeotKermackMcKendrickModel2023}. However, none of these models is as detailed as our presented model in terms of the compartments and the parameters used.

Our aim is to propose a comprehensive model where we explicitly model the number of patients needing hospitalization and intensive care, using transition distributions for the differently severe infection states,  and demonstrate how it can be used for realistic simulations. For the novel model, we extend a numerical scheme based on the approach introduced in~\cite{messina_non-standard_2022} to solve our model equations. We further show that by using a backwards finite difference scheme in an appropriate way, the original discretization scheme can be replaced equivalently by a computationally much more efficient one. We prove certain important mathematical and biological properties of the numerical solution such as mass-conservation, positivity-preservation and convergence of solution elements to some final size. Furthermore, we validate our approach numerically by demonstrating the convergence rate. Eventually, we compare our model to an ODE model for epidemic outbreaks to demonstrate the advancement also in a realistic context. 

This paper is structured as follows. In~\cref{sec:IDE-SECIR}, we introduce a novel IDE-SECIR-type model. In~\cref{sec:discretization}, we extend the nonstandard numerical scheme to solve our model equations. Thereby, we provide two discretizations which are equivalent under certain conditions and finally, we show how the discretized scheme can be initialized from reported case data. In the theoretical results in~\cref{sec:theory}, we prove that the discretization scheme is mass-conserving and preserves positivity of the solution and some statements on the discrete final size. In~\cref{sec:num}, we show the convergence of the numerical scheme and compare the IDE-based model to a corresponding ODE-based model with respect to behavior at change points, as well as their prediction of infection dynamics of COVID-19. Finally, we will discuss our results, present limitations and provide a conclusion.

\section{A novel SECIR-type IDE-based model}\label{sec:IDE-SECIR}\FloatBarrier
In this section, we introduce a generalized age-of-infection SECIR-type model allowing for a-, pre- and symptomatic transmission. The model is formulated using integro-differential equations, which is why we call it SECIR-type IDE-based model, or simply IDE model.
\begin{figure}
\centering
\scalebox{0.9}{
 \tikzstyle{roundbox}=[draw, fill=teal!20, rounded corners=10pt, minimum height =2.1cm, minimum width =3cm,text width = 2.5cm, align = center ]
	\tikzstyle{arrow}=[->, black, thick, text = black]
\begin{tikzpicture}[scale=0.5,every node/.style={scale=0.8}]<1
	\node [roundbox] (s) {Susceptible \\ $\mathbf{S}$};
	\node [roundbox, right = 3cm of s] (e) {Exposed, \\ Not Infectious \\ $\mathbf{E}$};
	\node [roundbox, right = 3cm of e,fill=red!20] (c) {Infectious, \\ No Symptoms \\  $\mathbf{C}$};
	\node [roundbox, below = 1.4cm of c,fill=red!20] (i) {Infectious, Symptoms\\  \textbf{$\mathbf{I}$}};
	\node [roundbox, below = 1.4cm of i] (h) {Infected, Severe \\ \textbf{$\mathbf{H}$}};
	\node [roundbox, left = 3cm of h] (u) {Infected, Critical \\ 
		\textbf{$\mathbf{U}$}};
	\node [roundbox, below = 1.4cm of s] (r) {Recovered \\ $\mathbf{R}$};
 \coordinate[xshift=-1mm,yshift=+1mm] (faker) at (r.south east);
	\node [roundbox, left = 3cm of u] (d) {Dead \\ $\mathbf{D}$};

	\draw [arrow] (s) -- node [above, font=\Large] {$\sigma_{S}^{E}$} (e);
	\draw [arrow] (e) -- node [above, font=\Large] {$\sigma_{E}^{C}$}(c);
	\draw [arrow] (c) -- node [right, font=\Large] {$\sigma_{C}^{I}$} (i);
	\draw [arrow] (c) -- node [right,xshift=0.9cm, font=\Large] {$\sigma_{C}^{R}$} (r);
	\draw [arrow] (i) -- node [above, font=\Large] {$\sigma_{I}^{R}$}(r);
	\draw [arrow] (i) -- node [right, font=\Large] {$\sigma_{I}^{H}$} (h);
	\draw [arrow] (h) -- node [below, font=\Large] {$\sigma_{H}^{U}$} (u);
	
	\draw [arrow] (h) -- node [above, font=\Large] {$\sigma_{H}^{R}$} (r);
	\draw [arrow] (u) -- node [below, font=\Large] {$\sigma_{U}^{D}$} (d);
	\draw [arrow] (u.north) -- node [below,xshift=-0.2cm, font=\Large] {$\sigma_{U}^{R}$} (faker);
\end{tikzpicture}}
 \caption{\textbf{Structure of the IDE model.} Schematic representation of the compartments and the transitions between the compartments in the
IDE model. The states in which
individuals are infectious are highlighted in red.}
\label{fig:IDESECIR}
\end{figure}
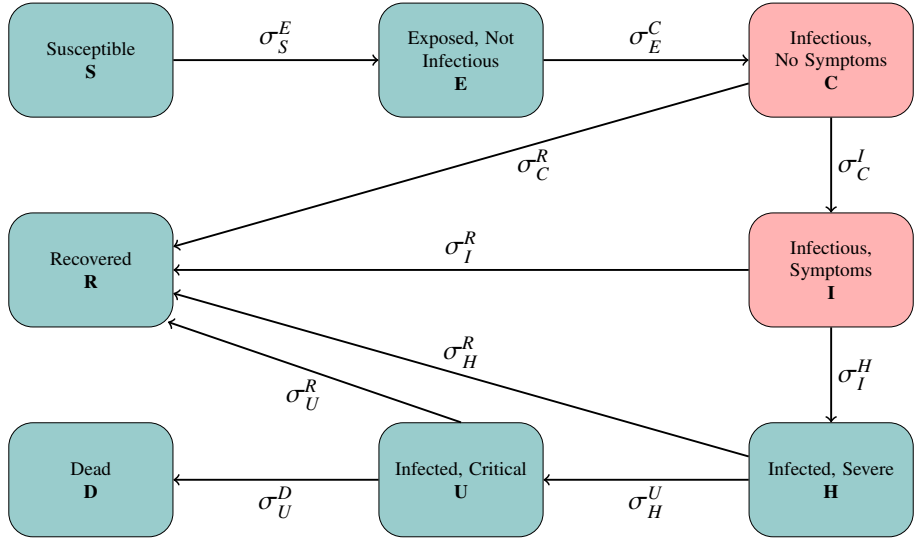

The model uses eight compartments, namely \textit{Susceptible}~($S$), for people who have not yet been infected with the disease and can therefore still be infected; \textit{Exposed}~($E$), who are infected but are not yet infectious to others; \textit{Carrier}~($C$), who are infected and infectious but do not show symptoms (they may be pre- or asymptomatic); \textit{Infected}~($I$), who are infected, infectious and show symptoms; \textit{Hospitalized}~($H$), who experience a severe development of the disease; \textit{In Intensive Care Unit}~($U$); \textit{Dead}~($D$); and \textit{Recovered}~($R$), who have gained full immunity. We define $\mathcal{Z}:=\{S, E, C, I, H, U, R, D\}$ as the set of compartments. For simplifying the notation, we use $Z\in\mathcal{Z}$ as indices to state-specific parameters as well as to denote the number of individuals $Z(t)$ at the particular infection state $Z$ at time $t$. The flow chart of the model is given in \cref{fig:IDESECIR}. 

We further define the number of people who enter a specific compartment $z_2\in\mathcal{Z}$ from $z_1\in\mathcal{Z}$ at time $x$ by $\sigma_{z_1}^{z_2}(x)$. These \textit{transitions} or \textit{flows} are only defined for pairs of consecutive compartments, as shown in~\cref{fig:IDESECIR}. Based on the transitions $\sigma_{z_1}^{z_2}(x)$, we can define our model equations for the compartments as follows
\begingroup
\allowdisplaybreaks[4]
\begin{align}\label{eq:IDESECIR}
\begin{aligned}
    S'(t) &= -\frac{S(t)}{N-D(t)}\, \phi(t)\,  \int_{-\infty}^t \AW{\xi_C(t-x) \, \rho_C(t-x)}\left(\mu_C^I\,\gamma_C^I(t-x)+\left(1-\mu_C^I\right)\gamma_C^R(t-x)\right) \sigma_E^C(x) \\
        & \hspace{3.05cm}  + \AW{\xi_I(t-x)\,\rho_I(t-x)} \left(\mu_I^H\,\gamma_I^H(t-x)+\left(1-\mu_I^H\right)\gamma_I^R(t-x)\right) \sigma_C^I(x) \ \d x ,\\
    E(t) &=  \int_{-\infty}^t \gamma_E^C(t-x)\, \sigma_S^E(x) \ \d x ,\\ 
    C(t) &= \int_{-\infty}^t \left(\mu_C^I\,\gamma_C^I(t-x)+\left(1-\mu_C^I\right)\gamma_C^R(t-x)\right) \sigma_E^C(x) \ \d x ,\\ 
    I(t) &= \int_{-\infty}^t \left(\mu_I^H\,\gamma_I^H(t-x)+\left(1-\mu_I^H\right)\gamma_I^R(t-x)\right) \sigma_C^I(x)\ \d x ,\\ 
    H(t) &= \int_{-\infty}^t \left(\mu_H^U\, \gamma_H^U(t-x) + \left(1-\mu_H^U\right) \gamma_H^R(t-x)\right) \sigma_I^H(x) \ \d x ,\\ 
    U(t) &= \int_{-\infty}^t  \left(\mu_U^D\, \gamma_U^D(t-x) + \left(1- \mu_U^D\right)\gamma_U^R(t-x)\right) \sigma_H^U(x) \ \d x ,\\ 
    R(t) &= \int_{-\infty}^t  \sigma_C^R(x) + \sigma_I^R(x)+ \sigma_H^R(x)+ \sigma_U^R(x)\ \d x ,\\
    D(t) &= \int_{-\infty}^t \sigma_U^D(x) \ \d x.
\end{aligned}
\end{align}
\endgroup
The total population (including deaths) is denoted by $N:=\sum_{Z\in\mathcal{Z}} Z(t)$ and is supposed to be constant over time. Additionally, $\phi(t)\geq0$ refers to the (bounded) average daily contacts of a person at time $t$ and $\AW{\rho_C(\tau)}\in[0,1]$ \AW{and $\rho_I(\tau)\in[0,1]$ are} the average transmission probabilit\AW{ies  
at infection age $\tau$ of individuals in compartments $C$ and $I$, respectively}. The parameters $\xi_C(\tau)\in[0,1]$ and $\xi_I( \tau)\in[0,1]$ denote the mean proportion of individuals in compartments $C$ and $I$, respectively, that are not isolated if they have been in the respective compartments for time $\tau$. The parameters $\mu_{z_1}^{z_2}\in[0,1]$ are defined as the expected proportion\AW{s} of people which move from compartment $z_1\in\mathcal{Z}$ to $z_2\in\mathcal{Z}$ in the course of their disease. The expression $\gamma_{z_1}^{z_2}(\tau)$ with $\gamma_{z_1}^{z_2}: \mathbb{R}\to[0,1]$ denotes the expected proportion of individuals who are still in compartment $z_1$ on $\tau$ days after entering this compartment and who will eventually move to compartment $z_2$ in the course of the disease. For theoretical purposes,\AW{we assume that these functions have the following properties:
\begin{itemize}
    \item The functions $\gamma_{z_1}^{z_2}$ are continuously differentiable on $(0,\infty)$.
    \item They are monotonically decreasing with
    $\gamma_{z_1}^{z_2}(\tau)=1$ for $\tau\leq0$; in particular, it holds that $\gamma_{z_1}^{z_2}(0)=1$.
    \item We assume that the expected stay time in compartment $z_1$ is finite, i.e. that $\int_0^\infty \gamma_{z_1}^{z_2}(\tau)\,\d\tau < \infty$.
\end{itemize}
}
\AW{These assumptions imply that $\lim_{\tau\rightarrow\infty}\gamma_{z_1}^{z_2}(\tau)=0$, i.e., that individuals will eventually leave compartment $z_1$.}
The definition of $\gamma_{z_1}^{z_2}$ together with the above properties implies that there exists a cumulative distribution function (CDF) $F_{z_1}^{z_2}(\tau):=1-\gamma_{z_1}^{z_2}(\tau)$ for $\tau\in\mathbb{R}$ which describes the distribution of the stay time. Consequently, $-{\gamma_{z_1}^{z_2}}':\mathbb{R} \to\mathbb{R}$ is a probability density function (PDF) that is continuous on $(0, \infty)$ and satisfies $-{\gamma_{z_1}^{z_2}}'(\tau)=0$ for $\tau <0$. \AW{Note that since $\gamma_{z_1}^{z_2}$ is monotonically decreasing, it holds that $-{\gamma_{z_1}^{z_2}}'$ is nonnegative, i.e., $-{\gamma_{z_1}^{z_2}}'(\tau)\geq 0$ for all $\tau\in\mathbb{R}$.}

The meaning of the parameters is summarized in~\cref{tab:parameters}. In order to fully describe the model, it is necessary to provide formulas for the transitions $\sigma_{z_1}^{z_2}$ which we will derive in the following.
\begin{table}
    \centering
    \begin{tabular}{c l}
     \hline
     Parameter &  Description\\
     \hline
     $N$ & Total population size.\\
      $\phi(t)$ & Daily contacts at simulation time $t$.\\
     \AW{$\rho_C(\tau)$} & \AW{Transmission risk of Carrier individuals on contact at infection age $\tau$.}\\
     \AW{$\rho_I(\tau)$} & \AW{Transmission risk of Infected individuals on contact at infection age $\tau$.} \\
     \AW{$\xi_{C}(\tau)$} & Proportion of Carrier individuals not isolated \AW{at} infection age $\tau$.\\
     \AW{$\xi_{I}(\tau)$} & Proportion of Infected individuals not isolated \AW{at} infection age $\tau$.\\
     $\mu_{z_1}^{z_2}$& Expected probability of transition from compartment $z_1$ to $z_2$.\\
     $\gamma_{z_1}^{z_2}(\tau)$ & Expected proportion of individuals who will be in compartment $z_1$ \\
     & on $\tau$ days after entering $z_1$ and who will eventually move to compartment $z_2$.
\end{tabular}
\caption{\textbf{Description of the parameters used to define the IDE model.}}\label{tab:parameters}
\end{table}

We begin by observing that the following relations between the derivatives of the compartments and the flows should hold true. This is the case because the change of the compartment sizes is determined by the in- and outflow, which is given by the respective transitions. This leads to 
\begin{align}
\begin{aligned}\label{eq:IDEAblKomp}
    S'(t)&=-\sigma_{S}^{E}(t) ,\\
    E'(t)&=\sigma_{S}^{E}(t)-\sigma_{E}^{C}(t),\\
    C'(t)&=\sigma_{E}^{C}(t)-\sigma_{C}^{I}(t)-\sigma_{C}^{R}(t),\\
    I'(t)&=\sigma_{C}^{I}(t)- \sigma_{I}^{H}(t)-\sigma_{I}^{R}(t),\\
    H'(t)&= \sigma_{I}^{H}(t)-\sigma_{H}^{U}(t)-\sigma_{H}^{R}(t),\\
    U'(t)&=\sigma_{H}^{U}(t)-\sigma_{U}^{D}(t)-\sigma_{U}^{R}(t),\\
    R'(t)&=\sigma_{C}^{R}(t)+\sigma_{I}^{R}(t)+\sigma_{H}^{R}(t)+\sigma_{U}^{R}(t),\\
    D'(t)&=\sigma_{U}^{D}(t).
\end{aligned}
\end{align}
The derivative of one compartment is calculated by adding the incoming flows and subtracting the outgoing flows. Therefore, we can derive the formulas for the transitions by taking the derivative of the formulas for the compartments, as introduced in~\eqref{eq:IDESECIR}.

To demonstrate the relation between~\eqref{eq:IDESECIR} and~\eqref{eq:IDEAblKomp} and derive a formula for the corresponding flows, we only consider compartment $C$. The results for the remaining compartments and flows can be obtained by analogous computations.
A derivation of the equation for $C$ from~\eqref{eq:IDESECIR}, employing the Leibniz integral rule, yields the formula
\begin{align*}
\begin{aligned}
      C'(t) &= \frac{\mathrm{d}}{\mathrm{d}t}\left(\int_{-\infty}^t \left(\mu_C^I\,\gamma_C^I(t-x)+\left(1-\mu_C^I\right)\gamma_C^R(t-x)\right) \sigma_E^C(x) \ \d x\right)\\
     &=\mu_C^I\,\gamma_C^I(0)\,\sigma_E^C(t)  + \AW{\mu_C^I}\int_{-\infty}^t {\gamma_C^I}'(t-x)\,\sigma_E^C(x) \ \d x  \\
     &\hspace{0.2cm}+\left(1-\mu_C^I\right)\gamma_C^R(0)\, \sigma_E^C(t) + \AW{\left(1-\mu_C^I\right)}\int_{-\infty}^t {\gamma_C^R}'(t-x)\, \sigma_E^C(x) \ \d x.
\end{aligned}
\end{align*} Using the assumption $\gamma_C^R(0)=\gamma_C^I(0)=1$, we obtain 
\begin{align*}
      C'(t) 
     &=\sigma_E^C(t) + \AW{\mu_C^I}\int_{-\infty}^t {\gamma_C^I}'(t-x)\,\sigma_E^C(x) \ \d x+\AW{\left(1-\mu_C^I\right)}\int_{-\infty}^t{\gamma_C^R}'(t-x)\, \sigma_E^C(x) \ \d x.
\end{align*}
A comparison of the result with the proposed formula for $C'$ in~\eqref{eq:IDEAblKomp},
\begin{align*}
     C'(t) &=\sigma_E^C(t)-\sigma_C^{I}(t)-\sigma_C^R(t),
\end{align*} along with the incorporation of the definitions of the parameters and the transitions, allows for the assignment
\begin{align*}
     C'(t) 
     &=\sigma_E^C(t) \,+\, \underbrace{\AW{\mu_C^I}\int_{-\infty}^t \,{\gamma_C^I}'(t-x)\,\sigma_E^C(x) \ \d x}_{=:-\sigma_C^{I}(t)}\,+\,\underbrace{\AW{\left(1-\mu_C^I\right)}\int_{-\infty}^t{\gamma_C^R}'(t-x)\, \sigma_E^C(x) \ \d x}_{=:-\sigma_C^R(t)}.
\end{align*}
The application of this methodology to the remaining compartments yields
\begin{align}
\begin{aligned}\label{eq:flows}
    \sigma_{S}^{E}(t)=&-S'(t)=
        S(t)\,\AW{\lambda}(t) ,
    \quad && &\sigma_{E}^{C}(t)=&
        - \int_{-\infty}^{t} {\gamma_{E}^{C}}'(t-x)\,\sigma_{S}^{E}(x)\, \d x,\\
    \sigma_{C}^{I}(t)=&
        -\AW{\mu_{C}^{I}\,}\int_{-\infty}^{t} {\gamma_{C}^{I}}'(t-x)\,\sigma_{E}^{C}(x)\, \d x ,
    \quad && &\sigma_{C}^{R}(t)=&
        -\AW{\left(1-\mu_{C}^{I}\right)}\,\int_{-\infty}^{t}{\gamma_{C}^{R}}'(t-x)\,\sigma_{E}^{C}(x)\, \d x ,\\
    \sigma_{I}^{H}(t)=&
        -\AW{\mu_{I}^{H}\,}\int_{-\infty}^{t} {\gamma_{I}^{H}}'(t-x)\,\sigma_{C}^{I}(x)\,\d x ,
    \quad && &\sigma_{I}^{R}(t)=&
        -\AW{\left(1-\mu_{I}^{H}\right)}\int_{-\infty}^{t} {\gamma_{I}^{R}}'(t-x)\,\sigma_{C}^{I}(x)\, \d x ,\\
    \sigma_{H}^{U}(t)=&
        -\AW{\mu_{H}^{U}\,}\int_{-\infty}^{t} {\gamma_{H}^{U}}'(t-x)\,\sigma_{I}^{H}(x)\, \d x ,
    \quad && &\sigma_{H}^{R}(t)=&
        -\AW{\left(1-\mu_{H}^{U}\right)}\int_{-\infty}^{t} {\gamma_{H}^{R}}'(t-x)\,\sigma_{I}^{H}(x)\, \d x ,\\
    \sigma_{U}^{D}(t)=&
        -\AW{\mu_{U}^{D}\,}\int_{-\infty}^{t} {\gamma_{U}^{D}}'(t-x)\,\sigma_{H}^{U}(x)\, \d x ,
    \quad && &\sigma_{U}^{R}(t)=&
        -\AW{\left(1-\mu_{U}^{D}\right)}\int_{-\infty}^{t} {\gamma_{U}^{R}}'(t-x)\,\sigma_{H}^{U}(x)\, \d x,
        \end{aligned}
\end{align}
with a force of infection defined as 
\begin{align} \label{eq:forceofinfection}
\begin{aligned}
	\AW{\lambda}(t) &= \frac{\phi(t)}{N-D(t)} \, \int_{-\infty}^t \xi_C(t-x)\AW{\,\rho_C(t-x)} \left(\mu_C^I\,\gamma_C^I(t-x)+\left(1-\mu_C^I\right)\gamma_C^R(t-x)\right) \sigma_E^C(x)\\
     & \hspace{2.2cm}  + \xi_I(t-x)\AW{\,\rho_I(t-x)} \left(\mu_I^H\,\gamma_I^H(t-x)+\left(1-\mu_I^H\right)\gamma_I^R(t-x)\right) \sigma_C^I(x) \ \d x. \\
\end{aligned}
\end{align}

\begin{remark}
In~\ref{app:ide_ode}, we show by elementary math operations that the here introduced model is in fact a generalization of the ODE-SECIR-type model presented in~\cite{kuhn_assessment_2021}. The IDE model is formulated for general transition distributions, thus allowing for nonexponential stay times. 
\end{remark}

\section{A nonstandard discretization scheme} \label{sec:discretization}
In the previous section, we have introduced a generalization of a model for infectious disease dynamics from~\cite{kuhn_assessment_2021}. However, the transition to an integro-differential system means that standard numerical ODE solvers cannot be used. In a recent publication~\cite{messina_non-standard_2022}, the authors showed that the application of a trapezoidal rule to discretize the integrals in the model equations does not conserve important (biological) properties of the continuous system \AW{for discretization step sizes above a certain threshold}. The violated properties include the monotonicity of $S$ and the positivity of the force of infection term $\AW{\lambda}$. To address this issue, the authors proposed a nonstandard scheme to solve the $S$ equation of an SIR model based on integro-differential equations. In the following, we extend this formula to our model compartments and flows as defined in the previous section. In~\cref{sec:theory}, we will furthermore prove that the extended discretization scheme still conserves essential (biological) properties as desired.

For the discrete approximations to the compartments as defined in~\eqref{eq:IDESECIR} (or~\eqref{eq:IDEAblKomp}), the flows as defined in~\eqref{eq:flows} and the force of infection term as defined in~\eqref{eq:forceofinfection}, we use the $\widehat{x}$ notation for the respective discretized version of variable $x$. For a step size $\Delta t>0$, let $t_n := n\,\Delta t$ for $n \in \mathbb{Z}$ define a uniform mesh. We define $a \in \mathbb{Z}_{-}$ such that $t_a \leq \widehat{T}$, where $\widehat{T} \in \mathbb{R}_-$ is a time point at which no infections have occurred yet. Such a time point $\widehat{T}$ is assumed to be known.

For $n \in \mathbb{N}_0$, we define
\begin{align}\label{eq:sigmaSE_discrete}
    \widehat{\sigma}_S^E(t_{n+1})&=-\widehat{S}{'}(t_{n+1})=\widehat{S}(t_{n+1})\,\widehat{\AW{\lambda}}(t_{n}),
\end{align}
which is a right endpoint approximation in $\widehat{S}$ and a left endpoint approximation in $\widehat{\AW{\lambda}}$. This gives us a nonstandard numerical scheme, as proposed in~\cite{messina_non-standard_2022}. Together with a backwards finite difference scheme for $\widehat{S}{'}$,
\begin{align} \label{eq:S_backwardsdifference}
    \widehat{S}{'}(t_{n+1}) = \frac{\widehat{S}(t_{n+1}) - \widehat{S}(t_n)}{\Delta t},
\end{align}
we obtain
\begin{align}
    \widehat{S}(t_{n+1}) &= \frac{\widehat{S}(t_n)}{1+\Delta t\, \widehat{\AW{\lambda}}(t_n)}.\label{eq:diskretisierung_S}
\end{align}
The force of infection term is discretized as
\begin{align} 
 \begin{aligned}\label{eq:diskretisierung_phi}
    \widehat{\AW{\lambda}}(t_{n+1})&
    = 
    \frac{\phi(t_{n+1})}{N-\widehat{D}(t_{n})} \,  \,  \Delta t\sum_{i=a}^{n}
 \biggl( \xi_C(t_{n+1-i})\AW{\,\rho_C(t_{n+1-i})} \,\Big(\mu_C^I \, \gamma_C^I (t_{n+1-i}) + (1- \mu_C^I) \,\gamma_C^R(t_{n+1-i})\Big) \, \widehat{\sigma}_E^C(t_{i+1})
 \\
	&\hspace{2.65cm}
 +
	\xi_I(t_{n+1-i})\,\AW{\rho_I(t_{n+1-i})\,}\Big( \mu_I^H \, \gamma_I^H(t_{n+1-i}) + (1- \mu_I^H) \, \gamma_I^R(t_{n+1-i})\Big) \, \widehat{\sigma}_C^I(t_{i+1}) 
	\biggr),
 \end{aligned}
\end{align}
which is again a nonstandard discretization by using a rectangular rule with a right endpoint approximation in $\widehat{\sigma}_E^C$ and $\widehat{\sigma}_C^I$ and a left endpoint approximation in $\xi_C$, $\xi_I$\AW{, $\rho_C$, $\rho_I$} and ${\gamma_{z_1}^{z_2}}'$ for appropriate $z_1$ and $z_2$.
The integrals in the formulas for the remaining flows will be approximated analogously by employing a nonstandard rectangular rule. We furthermore replace the derivative ${\gamma_{z_1}^{z_2}}'$ by a numerical approximation ${{\widehat\gamma}{}_{z_1}^{z_2}}'$, yet to be defined. Additionally, using~\eqref{eq:diskretisierung_S}, this results in the discretized system of flows,

\begingroup
\allowdisplaybreaks[4]
\begin{align} 
\begin{aligned} \label{eq:diskretisierung_sigma}\widehat{\sigma}_S^E(t_{n+1})&=\frac{\widehat{S}(t_n)}{1+\Delta t\, \widehat{\AW{\lambda}}(t_n)}\,\widehat{\AW{\lambda}}(t_{n})  , \quad &&
     &\widehat{\sigma}_E^{C} (t_{n+1}) &= - \Delta t \,\sum_{i=a}^n {\widehat{\gamma}{}_{E}^{C}}'(t_{n+1-i})\,\widehat{\sigma}_S^E(t_{i+1}) , \\
     \widehat{\sigma}_C^{I} (t_{i+1}) &= - \,\AW{\mu_C^I\,}\Delta t\, \sum_{i=a}^n  \, {\widehat{\gamma}{}_C^{I}}'(t_{n+1-i})\,\widehat{\sigma}_E^{C} (t_{i+1}) , \quad &&
     &\
     \widehat{\sigma}_C^R(t_{n+1}) &= -\AW{\, \left(1-\mu_C^I\right)\,}\Delta t \sum_{i=a}^{n}  \,{\widehat{\gamma}{}_C^R}'(t_{n+1-i}) \,\widehat{\sigma}_E^C(t_{i+1})  , \\
     \widehat{\sigma}_I^H(t_{n+1}) &= -\AW{\,\mu_I^H\,}\Delta t \sum_{i=a}^{n} \,{\widehat{\gamma}{}_I^H}'(t_{n+1-i}) \,\widehat{\sigma}_C^I(t_{i+1}) ,\quad &&
     &\widehat{\sigma}_I^R(t_{n+1}) &= - \AW{\,  \left(1-\mu_I^H\right)\,}\Delta t \sum_{i=a}^{n}\,{\widehat{\gamma}{}_I^R}'(t_{n+1-i}) \,\widehat{\sigma}_C^I(t_{i+1}),\\
    \widehat{\sigma}_H^U(t_{n+1}) &=- \AW{\, \mu_H^U\,}\Delta t \sum_{i=a}^{n}\, {\widehat{\gamma}{}_H^U}'(t_{n+1-i}) \,\widehat{\sigma}_I^H(t_{i+1}), \quad &&
     &\widehat{\sigma}_H^R(t_{n+1}) &=- \AW{\, \left(1-\mu_H^U\right)\,}\Delta t \sum_{i=a}^{n}\, {\widehat{\gamma}{}_H^R}'(t_{n+1-i}) \,\widehat{\sigma}_I^H(t_{i+1}), \\
    \widehat{\sigma}_U^D(t_{n+1})  &= -  \AW{\, \mu_U^D \,}\Delta t\sum_{i=a}^{n}  \, {\widehat{\gamma}{}_U^D}'(t_{n+1-i}) \,\widehat{\sigma}_H^U(t_{i+1}), \quad &&
     &\widehat{\sigma}_U^R(t_{n+1})  &= - \AW{\, \left(1-\mu_U^D\right) \,}\Delta t \sum_{i=a}^{n} \,{\widehat{\gamma}{}_U^R}'(t_{n+1-i}) \, \widehat{\sigma}_H^U(t_{i+1}).
\end{aligned}
\end{align}
\endgroup
We observe that the incoming flow of the current time step is required in order to compute the subsequent flows. Therefore, it is essential to calculate the flows in the correct sequence. 

Once the flows for all time points have been computed, the values for the compartments can be obtained. In the next sections, two approaches to the discretization of the compartments will be introduced. The second scheme has the advantage that it is much more computationally efficient throughout the simulation and in~\cref{thm:discretizations_equivalent}, we show that these discretizations are equivalent if we approximate the derivatives of the functions $\gamma_{z_1}^{z_2}$ using a backwards difference scheme. 

\subsection{Nonstandard discretization of compartments based on integral formulation} \label{subsec:sum_discretization}

In this section, we will introduce a discretization for the compartments apart from $S$ by directly applying the nonstandard rectangular rules to the integral model formulation as given in~\eqref{eq:IDESECIR}. Since this scheme directly discretizes the integral formulation, we will denote it as \textit{sum discretization}.

We distinguish two cases. The first category includes transient or intermediate compartments that individuals are entering and leaving, i.e., $E$, $C$, $I$, $H$, and $U$. The second category comprises absorbing or dead-end compartments, as $R$ and $D$, where individuals enter but cannot exit.

In the first case, we approximate the integral term by once again employing the nonstandard numerical rectangular rule that we previously used in equation~\eqref{eq:diskretisierung_phi}. Here, we use a right approximation for the flows and a left approximation for $\gamma_{z_1}^{z_2}$. In the second case, we use a standard right rectangular rule, where we use a right approximation of the respective flows. With this, we obtain the discrete approximations for all remaining compartments by
\begingroup
\allowdisplaybreaks[4]
\begin{align}
\begin{aligned}\label{eq:diskret_sum}
    \widehat{E}(t_{n+1}) &= 
    \Delta t \sum_{i=a}^{n} \gamma_E^C(t_{n+1-i})\, \widehat{\sigma}_S^E(t_{i+1}) ,\\
    \widehat{C}(t_{n+1}) &= 
    \Delta t \sum_{i=a}^{n} \left(\mu_C^I\,\gamma_C^I(t_{n+1-i})+\left(1-\mu_C^I\right)\gamma_C^R(t_{n+1-i})\right) \widehat{\sigma}_E^C(t_{i+1}),\\
    \widehat{I}(t_{n+1}) &=
    \Delta t \sum_{i=a}^{n} \left(\mu_I^H\,\gamma_I^H(t_{n+1-i})+\left(1-\mu_I^H\right)\gamma_I^R(t_{n+1-i})\right) \widehat{\sigma}_C^I(t_{i+1}),\\
    \widehat{H}(t_{n+1}) &= \Delta t \sum_{i=a}^{n}\left(\mu_H^U \,\gamma_H^U(t_{n+1-i}) + \left(1-\mu_H^U\right) \gamma_H^R(t_{n+1-i})\right) \widehat{\sigma}_I^H(t_{i+1}),\\
    \widehat{U}(t_{n+1}) &= \Delta t \sum_{i=a}^{n} \left(\mu_U^D\, \gamma_U^D(t_{n+1-i}) + \left(1- \mu_U^D\right)\gamma_U^R(t_{n+1-i})\right) \widehat{\sigma}_H^U(t_{i+1}),\\
    \widehat{R}(t_{n+1}) &= \Delta t \sum_{i=a}^{n} \left(\widehat{\sigma}_C^R(t_{i+1}) + \widehat{\sigma}_I^R(t_{i+1}) + \widehat{\sigma}_H^R(t_{i+1}) + \widehat{\sigma}_U^R(t_{i+1})\right),\\
    \widehat{D}(t_{n+1}) &= \Delta t \sum_{i=a}^{n}\widehat{\sigma}_U^D(t_{i+1}).
    \end{aligned}
\end{align}
\endgroup

\subsection{Nonstandard discretization of compartments based on flows}\label{subsec:update_discretization}
Another way to discretize the formulas for the compartments apart from $S$ is based on the formulas for the derivatives of the compartments in~\eqref{eq:IDEAblKomp}. 
The evaluation of these equations at a discrete time point $t_{n+1}$, together with the approximation of the derivative of the compartments using a backwards difference scheme and discretizations~\eqref{eq:diskretisierung_sigma}, leads to the following discrete formulas 
\begingroup
\allowdisplaybreaks[4]
\begin{align}
\begin{aligned}
\label{eq:diskret_update}
    \widehat{E}(t_{n+1})&=  \widehat{E}(t_n) +\, \Delta t \, \widehat{\sigma}_{S}^{E}(t_{n+1})
                        -\, \Delta t \, \widehat{\sigma}_{E}^{C}(t_{n+1}),\\
    \widehat{C}(t_{n+1})&=  \widehat{C}(t_n) +\, \Delta t \, \widehat{\sigma}_{E}^{C}(t_{n+1})
                        -\, \Delta t \, \widehat{\sigma}_{C}^{I}(t_{n+1})
                        -\,\Delta t \, \widehat{\sigma}_{C}^{R}(t_{n+1}),\\
    \widehat{I}(t_{n+1})&=  \widehat{I}(t_n) +\, \Delta t \, \widehat{\sigma}_{C}^{I}(t_{n+1})
                        -\, \Delta t \, \widehat{\sigma}_{I}^{H}(t_{n+1})
                        -\, \Delta t \, \widehat{\sigma}_{I}^{R}(t_{n+1}),\\
    \widehat{H}(t_{n+1})&=\widehat{H}(t_n) +\, \Delta t \, \widehat{\sigma}_{I}^{H}(t_{n+1})
                        -\,\Delta t \, \widehat{\sigma}_{H}^{U}(t_{n+1})
                        -\, \Delta t \, \widehat{\sigma}_{H}^{R}(t_{n+1}),\\
    \widehat{U}(t_{n+1})&= \widehat{U}(t_n) +\, \Delta t \, \widehat{\sigma}_{H}^{U}(t_{n+1})
                        -\, \Delta t \, \widehat{\sigma}_{U}^{D}(t_{n+1})
                        -\, \Delta t \, \widehat{\sigma}_{U}^{R}(t_{n+1}),\\
    \widehat{R}(t_{n+1})&= \widehat{R}(t_n) +\, \Delta t \, \widehat{\sigma}_{C}^{R}(t_{n+1}) 
                        +\, \Delta t \, \widehat{\sigma}_{I}^{R}(t_{n+1}) 
                        +\, \Delta t \, \widehat{\sigma}_{H}^{R}(t_{n+1}) 
                        +\, \Delta t \, \widehat{\sigma}_{U}^{R}(t_{n+1}),\\
    \widehat{D}(t_{n+1})&= \widehat{D}(t_n) +\, \Delta t \, \widehat{\sigma}_{U}^{D}(t_{n+1}).
    \end{aligned}
\end{align}
\endgroup
It is assumed that the values of the compartments of the previous time step $t_{n}$ are known. With this knowledge, we can compute the compartments of the current time step $t_{n+1}$ by updating the previous compartment values with the flows of the current time step. This discretization scheme will be denoted \textit{update discretization} as it incrementally updates the compartment values in the simulation.

\subsection{Connection between sum and update discretizations}
In this section and the following theorem, we demonstrate that the sum discretization as outlined in~\cref{subsec:sum_discretization} and the update discretization as described in~\cref{subsec:update_discretization} are equivalent when approximating the derivatives of the functions $\gamma_{z_1}^{z_2}$ using a backwards difference scheme.

\begin{theorem}\label{thm:discretizations_equivalent}
    Let the derivative of $\gamma_{z_1}^{z_2}$ for appropriate combinations of $z_1,z_2\in\mathcal{Z}$ be approximated with a backwards difference scheme, i.e., let
    \begin{align}\label{eq:gamma_backwards_difference}
        {\widehat{\gamma}{}_{z_1}^{z_2}}'(t_{i+1}):=\frac{ \gamma_{z_1}^{z_2}(t_{i+1})- \gamma_{z_1}^{z_2}(t_{i})}{\Delta t}
    \end{align}
    for $i\in\mathbb{Z}$. Then the sum discretization of the compartments, as defined in~\eqref{eq:diskret_sum}, is equivalent to the update discretization, as described in~\eqref{eq:diskret_update}.
\end{theorem}
\begin{proof}
    We begin to show the statement for the compartment $C$. The sum discretization of $C$ as stated in~\eqref{eq:diskret_sum} is given by 
    \begin{align*}
      \widehat{C}(t_{n+1}) =\Delta t\,\sum_{i=a}^{n}\left(\mu_{C}^{I}\,\gamma_{C}^{I}(t_{n+1-i})+\left(1-\mu_{C}^{I}\right)\,\gamma_{C}^{R}(t_{n+1-i})\right)\,\widehat{\sigma}_{E}^{C}(t_{i+1}).
    \end{align*} Calculating the difference of two consecutive time steps using this formula yields
    \begin{align*}
         \widehat{C}(t_{n+1}) - \widehat{C}(t_{n})
         &=\Delta t\,\sum_{i=a}^{n}\left(\mu_{C}^{I}\,\gamma_{C}^{I}(t_{n+1-i})+\left(1-\mu_{C}^{I}\right)\,\gamma_{C}^{R}(t_{n+1-i})\right)\widehat{\sigma}_{E}^{C}(t_{i+1})\\
        &\hspace{0.35cm}-\Delta t\,\sum_{i=a}^{n-1}\left(\mu_{C}^{I}\,\gamma_{C}^{I}(t_{n-i})+\left(1-\mu_{C}^{I}\right)\,\gamma_{C}^{R}(t_{n-i})\right)\widehat{\sigma}_{E}^{C}(t_{i+1})\\
        &= \Delta t\,\left(\mu_{C}^{I} \, \gamma_{C}^{I}(t_0) +\left(1-\mu_{C}^{I}\right)\,\gamma_{C}^{R}(t_0)\right)\,\widehat{\sigma}_{E}^{C}(t_{n+1})\\
         &\hspace{1cm} + \Delta t\,\sum_{i=a}^{n}\mu_{C}^{I}\left(\gamma_{C}^{I}(t_{n+1-i})-\gamma_{C}^{I}(t_{n-i})\right)\,\widehat{\sigma}_{E}^{C}(t_{i+1}) \\
        &\hspace{1cm}+\Delta t\,\sum_{i=a}^{n}\left(1-\mu_{C}^{I}\right) \left(\gamma_{C}^{R}(t_{n+1-i})-\gamma_{C}^{R}(t_{n-i})\right)\,\widehat{\sigma}_{E}^{C}(t_{i+1}).
    \end{align*} 
    By using $\gamma_{C}^{I}(t_0)=\gamma_{C}^{I}(0)=1$ and $\gamma_{C}^{R}(t_0)=\gamma_{C}^{R}(0)=1$, we obtain
    \begin{align}
    \begin{aligned}
    \label{eq:diff_C_sum}
         \widehat{C}(t_{n+1})- \widehat{C}(t_{n})&=\Delta t\,\widehat{\sigma}_{E}^{C}(t_{n+1}) + \Delta t\,\sum_{i=a}^{n}\mu_{C}^{I}\,\left(\gamma_{C}^{I}(t_{n+1-i})-\gamma_{C}^{I}(t_{n-i})\right)\,\widehat{\sigma}_{E}^{C}(t_{i+1})\\
        &\hspace{2.cm}+\Delta t\,\sum_{i=a}^{n}\left(1-\mu_{C}^{I}\right)\,\left(\gamma_{C}^{R}(t_{n+1-i})-\gamma_{C}^{R}(t_{n-i})\right)\,\widehat{\sigma}_{E}^{C}(t_{i+1}).
        \end{aligned}
    \end{align}
    Using the formulas~\eqref{eq:diskretisierung_sigma} for the relevant transitions and applying a backwards difference scheme as in~\eqref{eq:gamma_backwards_difference}, we obtain
     \begin{align*}
        \widehat{\sigma}_{C}^{I}(t_{n+1})
        &=-\,\AW{\mu_{C}^{I}}\sum_{i=a}^{n}\, \left(\gamma_{C}^{I}(t_{n+1-i})-\gamma_{C}^{I}(t_{n-i})\right)  \widehat{\sigma}_{E}^{C}(t_{i+1})\\
        \intertext{and}
        \widehat{\sigma}_{C}^{R}(t_{n+1})
        &=-\,\AW{\left(1-\mu_{C}^{I}\right)}\sum_{i=a}^{n} \,\left(\gamma_{C}^{R}(t_{n+1-i})-\gamma_{C}^{R}(t_{n-i})\right) \widehat{\sigma}_{E}^{C}(t_{i+1}).   
    \end{align*}
    We insert these formulas into~\eqref{eq:diff_C_sum} and get
    \begin{align*}
          \widehat{C}(t_{n+1})- \widehat{C}(t_{n})&=  \Delta t\, \widehat{\sigma}_{E}^{C}(t_{n+1}) -\Delta t\, \widehat{\sigma}_{C}^{I}(t_{n+1})
                        -\Delta t\, \widehat{\sigma}_{C}^{R}(t_{n+1}).                
    \end{align*} Consequently, we have shown that the sum discretization expression for the compartment $C$ is equivalent to the update discretization from~\eqref{eq:diskret_update}.
    For $I$, $H$ and $U$ the statement can be shown analogously. The proof for $E$ is a simplification of the proof for $C$. The equivalence of both discretizations for compartment $D$ can directly be seen through
    \begin{align*}
        \widehat{D}(t_{n+1}) &= \Delta t \, \sum_{i=a}^{n}\widehat{\sigma}_U^D(t_{i+1}) =  \Delta t \sum_{i=a}^{n-1 }\widehat{\sigma}_U^D(t_{i+1}) + \Delta t \, \widehat{\sigma}_U^D(t_{n+1}) \\
        &= \widehat{D}(t_n)  + \Delta t \,\widehat{\sigma}_U^D(t_{n+1}).
    \end{align*}
    Analogously, we obtain the equivalence of both discretizations of $R$.
\end{proof}
For the remainder of this paper, we assume that all derivatives of $\gamma_{z_1}^{z_2}$ are approximated using the backwards difference scheme. This allows us to use the update and sum discretization equivalently.

\subsection{Initialization of the discretized system}\label{sec:initialization}
In order to simulate realistic scenarios, we need to feed \AW{reported} data into our model. Initialization is a nontrivial detail but often neglected to be described. In the IDE model, we consider past flows to determine current disease dynamics. Accordingly, we need knowledge of the transitions before
the start time of our simulation, which must be derived from the actual data. The initialization of the IDE model is more challenging than that of ODE models, where only the sizes of the compartments at the start of the simulation are necessary. In this section, we assume that daily updated data on cases and deaths are available. For an example on how reported data can be used in the case of COVID-19 in Germany, we refer to~\cref{sec:covid_scenario}.

For the sake of simplicity, we assume that testing is symptom-based. In particular, this means that we assume that only symptomatic cases are detected. We assume that there is no dark figure, meaning that all symptomatic cases are tested and reported -- otherwise the formulas given below could be scaled by a constant or time changing detection ratio. Furthermore, we assume that cases are reported as soon as individuals develop symptoms, or in other words as soon as they enter compartment $I$ -- again, otherwise, a particular delay or a delay distribution could be introduced. With these assumptions, the daily reported cases correspond to the number of transitions from $C$ to $I$ within one day. Analogously, we assume that all deaths are reported. Furthermore, in accordance with the possible flows in the model, it is assumed that individuals only die after they have been in the intensive care unit. 

In order to use the daily reported case data as input to our model, the model itself must be adapted in the following way. The current discretized model evaluates the transitions at discrete time points. However, the data provides information on the number of new infections or deaths occurring within a day, thus within a time interval. Consequently, we define the transitions $\widetilde{\sigma}_{z_1}^{z_2}(t_n)$ for matching $z_1,z_2\in\mathcal{Z}$, that describe the number of individuals transitioning from $z_1$ to $z_2$ within the time interval $(t_{n-1}, t_n]$. The transitions $\widetilde{\sigma}_{z_1}^{z_2}(t_n)$ are related to the discretized transitions $\widehat{\sigma}_{z_1}^{z_2}(t_n)$, as defined in~\cref{sec:discretization}, as follows:
\begin{align*}
    \widetilde{\sigma}_{z_1}^{z_2}(t_n) = \int_{t_{n-1}}^{t_n} \sigma_{z_1}^{z_2}(t) \ dt \approx \Delta t \, \widehat{\sigma}_{z_1}^{z_2}(t_n),
\end{align*}
where we approximate $\sigma_{z_1}^{z_2}$ as in~\cref{sec:discretization} by $\widehat{\sigma}_{z_1}^{z_2}$ on the interval $(t_{n-1}, t_n]$ and then approximate the integral with a right rectangular rule. With this approximation, we can replace all $\widehat{\sigma}_{z_1}^{z_2}(t_n)$ in the equations in~\cref{sec:discretization} by $\widehat{\sigma}_{z_1}^{z_2}(t_n)\approx\widetilde{\sigma}_{z_1}^{z_2}(t_n)\hspace{0.1em}\AW{/}\Delta t$.

The reported case data gives us values for $\widetilde{\sigma}_C^I(t_n)$, assuming an interval of one day, i.e., $\Delta t=1$. When using smaller time steps, we use linear interpolation to obtain intermediate values. 

Then, the obtained data allows for the calculation of subsequent flows following compartment $I$, using the equations of the sum discretization as presented in~\eqref{eq:diskretisierung_sigma}.

Now, all that is missing are the approximations for the transitions from the compartment $S$ to $E$ and from $E$ to $C$. Here, we use a simplification by using the mean stay time of individuals in $C$ who will transit to $I$, denoted by $T_C^I$, and shifting $\widetilde{\sigma}_C^I$ accordingly. This results in
\begin{align*}
\widetilde{\sigma}_E^C(t_i)=\frac{1}{\mu_C^I}\, \widetilde{\sigma}_C^I(t_i+[{T}_C^I]),
\end{align*}
where we round $T_C^I$ to the next multiple of $\Delta t$, $[{T}_C^I]$, and use that individuals transit from $C$ to $I$ with probability $\mu_C^I$. 
Similarly, we use the mean stay time of individuals in $E$, denoted by $T_E^C$, and shift the flow appropriately, resulting in
\begin{align*}
\widetilde{\sigma}_S^E(t_i)=\frac{1}{\mu_C^I} \widetilde{\sigma}_C^I(t_i+[T_E^C+T_C^I]),
\end{align*} where we round $T_E^C+T_C^I$ again to the next multiple of $\Delta t$, $[T_E^C+T_C^I]$. Using the results for $\widetilde{\sigma}_S^E(t_i)$, we can also compute the flow from $C$ to $R$ as in equation~\eqref{eq:diskretisierung_sigma}. 

Using this scheme, all flows needed to initialize the model can be derived from the reported data.
Note that depending on the distribution that is chosen for the transitions, the support of the respective survival function $\gamma_{z_1}^{z_2}$ may be large to infinite. However, the assumption that a single first index case appeared at $-\infty<a<0$ means that $\gamma_{z_1}^{z_2}$ only needs to be evaluated for flows down to $a$. 

In the application, we numerically only evaluate $\gamma_{z_1}^{z_2}$ on the interval $(0, q)$, where $q \in \mathbb{R}_+$ is set such that $\gamma_{z_1}^{z_2}(q) > \varepsilon$ for some appropriately selected $\varepsilon >0$. 

Finally, we set the number of deaths at the simulation start which we assume to be known. 

\section{Theoretical properties of the numerical solution}\label{sec:theory}
In this section, we will show that the discretization scheme(s) preserves important properties of the disease dynamics, such as the positivity of all components and flows. Note that, with the assumption of~\cref{thm:discretizations_equivalent}, we can use both discretization schemes equivalently. In the course of this section, we will make meaningful assumptions regarding transition distributions and parameters. 

We begin with three lemmas that are fundamental to the subsequent theorems. The initial proposition demonstrates that the discrete scheme is mass-conserving.

\begin{lemma} \label{lem:massenerhaltung}
   Let ${\widehat{\gamma}{}_{z_1}^{z_2}}'$ for suitable combinations $z_1,z_2\in\mathcal{Z}$ be approximated as in~\cref{thm:discretizations_equivalent}. Furthermore, let the sum of all compartments at time $t_0$ be equal to the total population
    \begin{align*}
        \widehat{S}(t_0) + \widehat{E}(t_0) + \widehat{C}(t_0) + \widehat{I}(t_0) + \widehat{H}(t_0) + \widehat{U}(t_0) + \widehat{R}(t_0) + \widehat{D}(t_0) = N.
    \end{align*}
    
    Then, the sum of all compartments is equal to the total population at all subsequent time points $t_n, n\in\mathbb{N}$. 
\end{lemma}
\begin{proof}
    We show the statement by induction. By assumption, the statement holds for $t_0$.

    Now assume that the statement holds for $t_n$. The number of individuals in compartment $S$ at $t_{n+1}$ given by~\eqref{eq:diskretisierung_S} is 
    \begin{align*}
        \widehat{S}(t_{n+1}) &= \widehat{S}(t_n) - \Delta t\,\widehat{\sigma}_S^E(t_{n+1});
    \end{align*}
    cf.~\eqref{eq:sigmaSE_discrete} and~\eqref{eq:S_backwardsdifference}.
    By inserting this formula and using  of~\eqref{eq:diskret_update}, we obtain by summation that
    \begin{align*}
        \widehat{S}(t_{n+1}) &+ \widehat{E}(t_{n+1}) + \widehat{C}(t_{n+1}) + \widehat{I}(t_{n+1}) + \widehat{H}(t_{n+1}) + \widehat{U}(t_{n+1}) + \widehat{R}(t_{n+1}) + \widehat{D}(t_{n+1})\\
        = \, &\widehat{S}(t_{n}) + \widehat{E}(t_{n}) + \widehat{C}(t_{n}) + \widehat{I}(t_{n}) + \widehat{H}(t_{n}) + \widehat{U}(t_{n}) + \widehat{R}(t_{n}) + \widehat{D}(t_{n})\\
        = \, &N.
    \end{align*}
    This concludes our proof. 
\end{proof}

The following lemma demonstrates that, if the assumptions are satisfied, not all individuals will die. 
\begin{lemma}\label{lem:not_all_individuals_die}
    Let
\begin{align}
    \mu_C^I \, \mu_I^H \, \mu_H^U \, \mu_U^D <1 \label{eq:condition_notalldead_lemma}
\end{align}
as well as $\Delta t >0$. \AW{Assume that nonnegative transitions $\widehat{\sigma}_{z_1}^{z_2}(t_i)$ for $i \in \{a, \dots, n\}$ are given \AW{for suitable combinations $z_1,z_2\in\mathcal{Z}$} for some $n \in\mathbb{N}_0$ that is fixed but arbitrary.}

In addition, we assume
\begin{align*}
    \widehat{D}(t_0) < N
\end{align*}
as well as
\begin{align*}
    \widehat{S}(t_0) + \widehat{E}(t_0) + \widehat{C}(t_0) + \widehat{I}(t_0) + \widehat{H}(t_0) + \widehat{U}(t_0) + \widehat{R}(t_0) + \widehat{D}(t_0) = N.
\end{align*}
Furthermore, let
\begin{align}\label{eq:all_compartments_positive}
    \widehat{Z}(t_{n}) \geq 0
\end{align}
for $\widehat{Z} \in \widehat{\mathcal{Z}}:=\{\widehat S,\widehat E,\widehat C,\widehat I,\widehat H,\widehat U,\widehat R, \widehat D\}$.

Then it holds that
\begin{align*}
    \widehat{D}(t_\AW{n}) < N.
\end{align*}
\AW{This statement also holds as $n$ approaches infinity, i.e.,
\begin{align*}
    \lim_{n\to\infty} \widehat{D}(t_{n}) < N.
\end{align*}}

\end{lemma}

\begin{proof}
    Assumption~\eqref{eq:condition_notalldead_lemma} implies that there is at least one $Z\in\mathcal{Z}$ such that $\mu_{Z}^{R}>0$. Without loss of generality, we assume that it holds $\mu_C^R = 1-\mu_C^I>0$. Given that all transitions are nonnegative, it can be concluded from the equation for $\widehat{R}$ in~\eqref{eq:diskret_sum} that
    \begin{align}\label{eq:R_bigger_sigmaCR}
        \widehat{R}(t_\AW{{n}}) \geq \Delta t \, \widehat{\sigma}_C^R(t_\AW{{n}}).
    \end{align}
    \AW{By~\eqref{eq:diskretisierung_sigma}, $ \widehat{\sigma}_C^R(t_{n})$ is defined} as
    \begin{align*}
        \widehat{\sigma}_C^R(t_\AW{{n}}) = -\Delta t \sum_{i=a}^{n\AW{-1}}  \left(1-\mu_C^I\right)\,{\widehat{\gamma}{}_C^R}'(t_\AW{{n-i}}) \,\widehat{\sigma}_E^C(t_{i+1}).
    \end{align*}
    We distinguish between two cases:
    \begin{enumerate}
        \item Assume that there exists an index $i \in \{a,\dots,n\AW{-1}\}$ such that 
        \begin{align*}
            -{\widehat{\gamma}{}_C^R}'(t_\AW{{n-i}}) \,\widehat{\sigma}_E^C(t_{i+1})>0.
        \end{align*} This directly implies $\widehat{R}(t_\AW{n})>0$, see equation~\eqref{eq:R_bigger_sigmaCR}. The application of~\cref{lem:massenerhaltung} together with assumption~\eqref{eq:all_compartments_positive} leads to the conclusion that it holds $\widehat{D}(t_\AW{n}) <N$.

        \AW{ When $n$ is approaching infinity, we can apply the same argument and obtain
        \begin{align*}
            \lim_{n\to\infty} \widehat{R}(t_{n}) > 0.
        \end{align*}
        This implies
        \begin{align*}
            \lim_{n\to\infty} \widehat{D}(t_{n}) \leq N - \lim_{n\to\infty} \widehat{R}(t_{n}) < N.
        \end{align*}
        }
        
        \item Assume that there exists no index $i \in \{a,\dots,n\AW{-1}\}$ such that
        \begin{align*}
            -{\widehat{\gamma}{}_C^R}'(t_\AW{{n-i}}) \,\widehat{\sigma}_E^C(t_{i+1})>0.
        \end{align*} 

        \begin{enumerate}[label=(\roman*)]
            \item Either there exists some $i$ such that $\widehat{\sigma}_E^C(t_{i+1})>0$. \AW{This implies that} $-{\widehat{\gamma}{}_C^R}'(t_{\AW{n-i}})=0$. 
            \AW{We observe that $\gamma_C^R$ is not constant on $(0,\infty)$. This implies that there exists at least one $j\in\mathbb{N}$ such that $-{\gamma_C^R}{}'(t_j)>0$. For simplicity, we only consider the case with one such index $j$; in the case of multiple such indices, one can extend the argument analogously.

            In the here considered case of $\widehat{\sigma}_E^C(t_{i+1})>0$ and $-{\widehat{\gamma}{}_C^R}'(t_{n-i})=0$, it holds that $n-i\neq j$. This means that} there is some nontrivial flow from $E$ to $C$ at $t_{i+1}$ but the respective individuals are still \AW{remaining} in compartment $C$ \AW{or they have already transitioned to $R$} at time $t_\AW{n}$ as the approximated change from $C$ to $R$ is zero. \AW{More precisely, either it holds $n-i<j$ which implies $\widehat{C}(t_{n})>0$, or else it holds $n-i>j$ which implies $\widehat{R}(t_{n})>0$. W}ith~\eqref{eq:all_compartments_positive} and~\cref{lem:massenerhaltung} it follows that $\widehat{D}(t_\AW{n})<N$.

            \AW{As $n$ approaches infinity, the inequality $n-i<j$ does not hold for $i\in\{a,\dots,n-1\}$ for $n$ large enough. Hence, if it holds $\widehat{\sigma}_E^C(t_{i+1})>0$ and $-{\widehat{\gamma}{}_C^R}'(t_{n-i})=0$ for $i\in\{a,\dots,n-1\}$, we have $n-i>j$ and thus $\widehat{R}(t_{n})>0$. In other words, if $n$ is large enough, all individuals have already transitioned to compartment $R$ at time $t_{n}$
            which implies
            \begin{align*}
            \lim_{n\to\infty} \widehat{R}(t_{n})  > 0,
            \end{align*}
            and thus
            \begin{align*}
                \lim_{n\to\infty} \widehat{D}(t_{n}) < N.
            \end{align*}}

            \item If there exists no such $i$, then we have $\widehat{\sigma}_E^C(t_{i+1})=0$ for all $i \in \{a,\dots,n\AW{-1}\}$. This implies that there has been no flow from $E$ to $C$ until $t_\AW{n}$. Thus, we have
            \begin{align*}
                \widehat{Z}(t_\AW{n}) = \widehat{Z}(t_0)
            \end{align*}
            for all subsequent compartments $\widehat Z \in \{\widehat C,\widehat  I,\widehat  H,\widehat  U,\widehat  R,\widehat  D\}$. By assumption, it holds $\widehat{D}(t_0)<N$ and consequently $\widehat{D}(t_\AW{n})<N$.

            \AW{When $n$ is approaching infinity, we can apply the same argument and obtain
            \begin{align*}
                \lim_{n\to\infty} \widehat{D}(t_{n}) = D(0) < N.
            \end{align*}}
        \end{enumerate} 
    \end{enumerate}
\end{proof}

\noindent To proof~\cref{thm:eigenschaften_flows_kompartimente}, one more lemma is needed.

\begin{lemma}\label{lem:bounded_series}
    Let $\Delta t >0$ be arbitrary but fixed and ${\widehat{\gamma}{}_{z_1}^{z_2}}'$ be the backwards difference approximation of ${\gamma_{z_1}^{z_2}}'$. Then there exists an $M\in\mathbb{R}$ such that
    \begin{align*}
        \sum_{n=0}^\infty -{\widehat{\gamma}{}_{z_1}^{z_2}}'(t_{n+1}) \leq M < \infty.
    \end{align*} 
\end{lemma}
\begin{proof}
    We start by using the backwards difference approximation as in~\eqref{eq:gamma_backwards_difference} and obtain
    \begin{align*}
        \sum_{n=0}^\infty -{\widehat{\gamma}{}_{z_1}^{z_2}}'(t_{n+1}) &= \lim_{K\to\infty} \sum_{n=0}^K -{\widehat{\gamma}{}_{z_1}^{z_2}}'(t_{n+1})\\
        &= \lim_{K\to\infty} \sum_{n=0}^K - \frac{\gamma_{z_1}^{z_2}(t_{n+1}) - \gamma_{z_1}^{z_2}(t_n)}{\Delta t}\\
        &= -\frac{1}{\Delta t} \left(\lim_{K\to\infty} \gamma_{z_1}^{z_2}(t_\AW{{K+1}}) - \gamma_{z_1}^{z_2}(t_0)\right)\\
        &= \frac{1}{\Delta t} \leq M < \infty,
    \end{align*}
    where we used that $1-\gamma_{z_1}^{z_2}$ is a CDF \AW{which implies that} $\gamma_{z_1}^{z_2}$ converges to zero.
\end{proof}

The following theorem demonstrates that the discretization scheme preserves important properties of the disease dynamics regarding positivity and boundedness. Our theorem generalizes the corresponding part of Theorem~$3.4$ of~\cite{messina_non-standard_2022}.

\begin{theorem}  \label{thm:positivity_boundedness}
    Let
\begin{align*}
    \mu_C^I \, \mu_I^H \, \mu_H^U \, \mu_U^D <1 
\end{align*}
and $\Delta t >0$ be arbitrary but fixed. Assume that nonnegative transitions $\widehat{\sigma}_{z_1}^{z_2}(t_\AW{i})$ for $\AW{i} \in \{a, \dots, 0\}$ are given \AW{for suitable combinations $z_1,z_2\in\mathcal{Z}$}.
Let it hold that
\begin{align*}
    0 \leq \widehat{Z}(t_0) \leq N
\end{align*}
for $\widehat{Z} \in \widehat{\mathcal{Z}}\backslash \{\widehat{D}\}$
as well as
\begin{align} \label{eq:D0_smaller_N}
    0 \leq \widehat{D}(t_0) < N
\end{align}
and
\begin{align*}
    \widehat{S}(t_0) + \widehat{E}(t_0) + \widehat{C}(t_0) + \widehat{I}(t_0) + \widehat{H}(t_0) + \widehat{U}(t_0) + \widehat{R}(t_0) + \widehat{D}(t_0) = N.
\end{align*}

Let $\widehat{\sigma}_{z_1}^{z_2}(t_{n})$, $\widehat{Z}(t_{n})$ with $\widehat{Z} \in \widehat{\mathcal{Z}}$, and $\widehat{\AW{\lambda}}(t_{n})$ for $n \in \mathbb{N}_0$ be the discrete solutions to our model as defined in~\eqref{eq:diskretisierung_S}, \eqref{eq:diskretisierung_phi},~\eqref{eq:diskretisierung_sigma} and~\eqref{eq:diskret_sum}.

Then the following statements hold for all $n \in \mathbb{N}_0$:
\begin{itemize}
    \item $\widehat{\sigma}_{z_1}^{z_2}(t_n)$ is nonnegative for all suitable combinations $z_1, z_2 \in \mathcal{Z}$,
    \item $0 \leq \widehat{Z}(t_n) \leq N$ for $\widehat{Z} \in  \widehat{\mathcal{Z}}$, in particular $\widehat{D}(t_n) <N$, and
    \item $\widehat{\AW{\lambda}}(t_n)$ is nonnegative.
\end{itemize}
    
\end{theorem}
\begin{proof}
We prove the statements by induction and start by showing the properties for $n=0$. 
    \begin{itemize}
        \item The values $\widehat{\sigma}_{z_1}^{z_2}(t_0)$ are nonnegative by assumption for all suitable $z_1, z_2 \in \mathcal{Z}$. In addition, the statement on the compartment sizes holds by assumption.
        \item As defined in~\eqref{eq:diskretisierung_phi}, $\widehat{\AW{\lambda}}(t_{0})$  depends on the transitions $\widehat{\sigma}_E^C(t_i)$ and $\widehat{\sigma}_C^I(t_i)$ for $i \in \{a, \dots, 0\}$ which are nonnegative by assumption. Assumption~\eqref{eq:D0_smaller_N} implies $N-\widehat{D}(t_0) > 0$. Since all other factors in $\widehat{\AW{\lambda}}(t_{0})$ are nonnegative as well, it follows that $\widehat{\AW{\lambda}}(t_{0}) \geq 0$.
    \end{itemize}        

Now assume that the properties hold for the time points $t_0, \dots, t_{n}$. 
    \begin{itemize}
        \item We argue analogously to~\cite{messina_non-standard_2022} that it holds
        \begin{align}
            \widehat{S}(t_{n+1}) = \frac{\widehat{S}(t_n)}{1+ \Delta t \, \widehat{\AW{\lambda}}(t_n)} \geq 0, \label{eq:S_nonnegative}
        \end{align}
        since by assumption $\widehat{S}(t_n) \geq 0$ and $\widehat{\AW{\lambda}}(t_n) \geq 0$. 
    \item This also leads to the bound
    \begin{align*}
        \widehat{\sigma}_S^E(t_{n+1}) = \widehat{S}(t_{n+1}) \ \widehat{\AW{\lambda}}(t_n) \geq 0.
    \end{align*}  
    By looking at the equations in~\eqref{eq:diskretisierung_sigma}, we observe that in consequence, all other transitions are nonnegative as well since $\gamma_{z_1}^{z_2}$ is monotonously decreasing and thus for the derivative it holds that ${\widehat{\gamma}{}_{z_1}^{z_2}}' \leq 0$.
    \item It follows with~\eqref{eq:diskret_sum} that all remaining approximations to compartments $\widehat{Z} \in \widehat{\mathcal{Z}}\backslash\{\widehat{S}\}$ are nonnegative at $t_{n+1}$ since all transitions as well as all $\gamma_{z_1}^{z_2}$ are nonnegative and we have $0\leq \mu_{z_1}^{z_2}\leq1$ for all valid transitions. 
    \item As we know from~\cref{lem:massenerhaltung} that 
    \begin{align*}
        \sum_{\widehat{Z} \in \mathcal{\widehat{Z}}}\widehat{Z}(t_{n+1})=N
    \end{align*} and since we obtained $\widehat{Z}(t_{n+1}) \geq 0$ for all $\widehat{Z}\in \mathcal{\widehat{Z}}$, we conclude that it holds $\widehat{Z}(t_{n+1}) \leq N$ for $\widehat{Z}\in  \mathcal{\widehat{Z}}$.
    \item Since all compartments at time point $t_{n+1}$ as well as the necessary transitions are nonnegative, we can apply ~\cref{lem:not_all_individuals_die} to obtain that $\widehat{D}(t_{n+1})<N$. 
    \item In particular, the previous observation implies $N-\widehat{D}(t_{n+1})>0$. Since all other factor\AW{s} in the force of infection term are nonnegative, we conclude that $\widehat{\AW{\lambda}}(t_{n+1}) \geq 0$.
    \end{itemize}
\end{proof}

In the following theorem, we will discuss how our numerical solutions behave in the limit. This is a generalization of the remaining statements of Theorem~$3.4$ in~\cite{messina_non-standard_2022}.
\begin{theorem}
\label{thm:eigenschaften_flows_kompartimente}
Let all the assumptions from~\cref{thm:positivity_boundedness} hold. 

\AW{Additionally, we assume the transitions $\widehat{\sigma}_{z_1}^{z_2}(t_\AW{i})$ for $\AW{i} \in \{a, \dots, 0\}$ are bounded for suitable combinations $z_1,z_2\in\mathcal{Z}$.}
Furthermore, we assume that there exists an $M<\infty$ \AW{such} that 
 \begin{align}\label{eq:assumption_xiC}
    \sum_{i=0}^\infty \AW{\xi_C(t_{i+1})\,\rho_C(t_{i+1})} \left(\mu_C^I \, \gamma_C^I (t_{i+1}) + (1- \mu_C^I) \,\gamma_C^R(t_{i+1})\right) \leq M < \infty 
\end{align}
and 
\begin{align}\label{eq:assumption_xiI}
    \sum_{i=0}^\infty \AW{\xi_I(t_{i+1})\,\rho_I(t_{i+1})}\left( \mu_I^H \, \gamma_I^H(t_{i+1}) + (1- \mu_I^H) \, \gamma_I^R(t_{i+1})\right) \leq M < \infty
\end{align}
for all $n \in \mathbb{N}_0$.
Note that this is trivially satisfied if $\xi_C$\AW{, }$\xi_I$\AW{, $\rho_C$, and $\rho_I$} have finite support in the infection age or vanish sufficiently fast.

Let $\widehat{\sigma}_{z_1}^{z_2}(t_{n})$, $\widehat{Z}(t_{n})$ with $\widehat{Z} \in \widehat{\mathcal{Z}}$, and $\widehat{\AW{\lambda}}(t_{n})$ for $n \in \mathbb{N}_0$ be the discrete solutions to our model as defined in~\eqref{eq:diskretisierung_S}, \eqref{eq:diskretisierung_phi},~\eqref{eq:diskretisierung_sigma} and~\eqref{eq:diskret_sum}.

Then the following statements hold:
\begin{enumerate}
    \item The sequence $(\widehat{S}(t_n))_{n \in \mathbb{N}_0}$ is nonincreasing and it holds
        \begin{align*}
            \lim_{n \to \infty} \widehat{S}(t_n) = \widehat{S}_{\infty}(\Delta t).
        \end{align*}
    \item The sequences $(\widehat{\sigma}_{z_1}^{z_2}(t_n))_{n \in \mathbb{N}_0}$ are bounded and it holds
        \begin{align*}
            \lim_{n \to \infty} \widehat{\sigma}_{z_1}^{z_2}(t_n) = 0.
        \end{align*}
    \item The sequence $(\widehat{\AW{\lambda}}(t_n))_{n \in \mathbb{N}_0}$ is bounded and it holds
        \begin{align*}
            \lim_{n \to \infty} \widehat{\AW{\lambda}}(t_n) = 0.
        \end{align*}
    \item The sequences $(\widehat{R}(t_n))_{n \in \mathbb{N}_0}$ and $(\widehat{D}(t_n))_{n \in \mathbb{N}_0}$ are nondecreasing and it holds
        \begin{align*}
            \lim_{n \to \infty} \widehat{R}(t_n) = \widehat{R}_\infty(\Delta t)\quad\text{and}\quad\lim_{n \to \infty} \widehat{D}(t_n) = \widehat{D}_\infty(\Delta t).
        \end{align*}
\end{enumerate}

\end{theorem}

\begin{proof}
    \begin{enumerate}
        \item\label{enum:proof1} Similar to~\cite{messina_non-standard_2022}, this property follows from~\cref{thm:positivity_boundedness}: 
            \begin{align*}
                1+ \Delta t \,\widehat{\AW{\lambda}}(t_n) \geq 1 \quad \Rightarrow \quad \widehat{S}(t_{n+1}) = \frac{\widehat{S}(t_n)}{1+ \Delta t\, \widehat{\AW{\lambda}}(t_n)} \leq \widehat{S}(t_n).
            \end{align*}
            The sequence $(\widehat{S}(t_n))_{n \in \mathbb{N}_0}$ is monotonously decreasing and bounded by zero from below, see~\eqref{eq:S_nonnegative}. Hence, the sequence is convergent and there exists $\widehat{S}_\infty(\Delta t)\geq0$ such that 
            \begin{align*}
                \lim_{n\to\infty} \widehat{S}(t_n) = \widehat{S}_\infty(\Delta t).
            \end{align*}

        \item\label{enum:proof2} With~\eqref{eq:sigmaSE_discrete} and~\eqref{eq:S_backwardsdifference}, it holds
            \begin{align*}
                \widehat{\sigma}_S^E(t_{n+1}) =  \frac{\widehat{S}(t_n) - \widehat{S}(t_{n+1})}{\Delta t}.
            \end{align*}
            For $n \to \infty$ it follows that       
            \begin{align}\label{eq:sigmaSE_converges_to_0}
                \lim_{n \to\infty} \widehat{\sigma}_S^E(t_{n+1}) = \lim_{n \to\infty} \frac{\widehat{S}(t_n) - \widehat{S}(t_{n+1})}{\Delta t}=0,
            \end{align}
            since $\Delta t$ is fixed and the sequence $(\widehat{S}(t_n))_{n \in \mathbb{N}_0}$ is convergent; see~item~\ref{enum:proof1} above.
            Since $(\widehat{\sigma}_S^E(t_n))_{n \in \mathbb{N}_0}$ converges, the sequence is also bounded.

            We continue by showing that
            \begin{align*}
                \widehat{\sigma}_E^C(t_{n+1}) = -\Delta t \sum_{i=a}^n {\widehat{\gamma}{}_{E}^{C}}'(t_{n+1-i})\ \widehat{\sigma}_S^E(t_{i+1})
            \end{align*}
            converges to zero as well for $n \to \infty$.
            To simplify the notation, we define
            \begin{align*}
                x_n :=\sum_{i=a}^n -{\widehat{\gamma}{}_{E}^{C}}'(t_{n+1-i}) \ \widehat{\sigma}_S^E(t_{i+1}) = \sum_{i=0}^{n-a} -{\widehat{\gamma}{}_{E}^{C}}'(t_{i+1}) \ \widehat{\sigma}_S^E(t_{n+1-i}).
            \end{align*}
            Let us recall here that $a<0$ and $-{\widehat{\gamma}{}_{E}^{C}}'(t_{i}) \geq 0$ for all $ i \in \mathbb{N}$ since $\gamma_E^C$ is monotonously decreasing. From~\cref{thm:positivity_boundedness}, we already know that $\widehat{\sigma}_S^E(t_{i})\geq 0$ for all $ i \in \mathbb{N}$. In addition, by~\cref{lem:bounded_series}, there exists $K\in\mathbb{R}$ such that
                    \begin{align}
                        \sum_{i=0}^{\infty} -{\widehat{\gamma}{}_{E}^{C}}'(t_{i+1}) \leq K < \infty. \label{eq:sigmaconv_assumption_gamma}
                    \end{align} 
            Since $\widehat{\sigma}_S^E(t_i)$ is bounded and nonnegative, it exists $0<L<\infty$ such that
            \begin{align}
                \widehat{\sigma}_S^E(t_{i}) < L \text{ for all } i\geq a.\label{eq:sigmaconv_conditionsigma}
            \end{align}
            Now, let $\varepsilon>0$ be arbitrary but fixed. We choose $N(\varepsilon)>0$ such that for all $Q > N(\varepsilon) + a$ it holds that
            \begin{align}
                \sum_{i=N(\varepsilon)}^{Q-a} - {\widehat{\gamma}{}_{E}^{C}}'(t_{i+1})  < \frac{\varepsilon}{2L} \label{eq:sigmaconv_proof1}.
            \end{align}
            Such an $N(\varepsilon)$ exists due to~\eqref{eq:sigmaconv_assumption_gamma}.
            
           We consider $x_Q$ for such a $Q > N(\varepsilon) + a$,
            \begin{align*}
                x_{Q}
                &= \sum_{i=0}^{Q-a} -{\widehat{\gamma}{}_{E}^{C}}'(t_{i+1}) \ \widehat{\sigma}_S^E(t_{Q+1-i}) \\ 
                &= \underbrace{\sum_{i=0}^{N(\varepsilon)-1} -{\widehat{\gamma}{}_{E}^{C}}'(t_{i+1}) \ \widehat{\sigma}_S^E(t_{Q+1-i})}_{=:\,S_1} + \underbrace{\sum_{i=N(\varepsilon)}^{Q-a} -{\widehat{\gamma}{}_{E}^{C}}'(t_{i+1}) \ \widehat{\sigma}_S^E(t_{Q+1-i})}_{=:\,S_2}.
            \end{align*}

            We will estimate the summands $S_1$ and $S_2$ of $x_Q$ separately. 
            
            For $S_2$, we observe that for any $Q>N(\varepsilon)+a$ it holds that
            \begin{align}
                S_2=\sum_{i=N(\varepsilon)}^{Q-a} -{\widehat{\gamma}{}_{E}^{C}}'(t_{i+1})\ \widehat{\sigma}_S^E(t_{Q+1-i}) < L\,    \sum_{i=N(\varepsilon)}^{Q-a} -{\widehat{\gamma}{}_{E}^{C}}'(t_{i+1}) < \frac{\varepsilon}{2} \label{eq:sigmaconv_zweiterteilSumme}
            \end{align}
            where we are using~\eqref{eq:sigmaconv_conditionsigma} and~\eqref{eq:sigmaconv_proof1}.

            Now we consider the summand
            \begin{align}
               S_1= \sum_{i=0}^{N(\varepsilon)-1} -{\widehat{\gamma}{}_{E}^{C}}'(t_{i+1})\ \widehat{\sigma}_S^E(t_{Q+1-i}). \label{eq:sigmaconv_firstpartofxQ}
            \end{align}Note that it holds $Q+1-i \in \{Q+2-N(\varepsilon), \dots, Q+1 \}$ for $i \in \{0, \dots , N(\varepsilon)-1\}$.

           We choose $Q_1(\varepsilon)$ such that for all $q > Q_1(\varepsilon) $ the estimate
            \begin{align}
                \widehat{\sigma}_S^E(t_q) < \frac{\varepsilon}{2K} \label{eq:sigmaconv_proof2}
            \end{align}
            is satisfied. There exists such a $Q_1(\varepsilon)$ since $\widehat{\sigma}_S^E$ converges to zero; see~\eqref{eq:sigmaSE_converges_to_0}.
            We set
            \begin{align}
                Q(\varepsilon) := Q_1(\varepsilon) + N(\varepsilon). \label{eq:sigmaconv_Qeps}
            \end{align}
            For all $Q>Q(\varepsilon)$, it holds that
            \begin{align*}
                Q+2-N(\varepsilon) > Q_1(\varepsilon)+N(\varepsilon) +2-N(\varepsilon) > Q_1(\varepsilon). 
            \end{align*}
            This implies
            \begin{align}
                \widehat{\sigma}_S^E(t_{Q+1-i}) < \frac{\varepsilon}{2K}, \label{eq:sigmaconv_condition_sigma_final}
            \end{align}
            for $i \in \{0, \dots , N(\varepsilon)-1\}$, compare~\eqref{eq:sigmaconv_firstpartofxQ} and~\eqref{eq:sigmaconv_proof2}.

            Finally, we obtain
            \begin{align}
                S_1 = \sum_{i=0}^{N(\varepsilon)-1} -{\widehat{\gamma}{}_{E}^{C}}'(t_{i+1})\,\underbrace{\vphantom{\displaystyle\sum_{i=1}^N}\widehat{\sigma}_S^E(t_{Q+1-i})}_{<\frac{\varepsilon}{2K}} \leq  \frac{\varepsilon}{2K}\, \underbrace{\sum_{i=0}^{N(\varepsilon)-1} -{\widehat{\gamma}{}_{E}^{C}}'(t_{i+1})}_{\leq K} \leq  \frac{\varepsilon}{2K} \, K =  \frac{\varepsilon}{2}, \label{eq:sigmaconv_ersterteilSumme}
            \end{align}
            where we first use~\eqref{eq:sigmaconv_condition_sigma_final} and afterward~\eqref{eq:sigmaconv_assumption_gamma}.

            For any $Q>Q(\varepsilon)$, it holds $Q>N(\varepsilon) +a$ since $a<0$ and~\eqref{eq:sigmaconv_Qeps}.
            Hence, for any $Q>Q(\varepsilon)$ we obtain
            \begin{align*}
                x_Q&= \sum_{i=0}^{Q-a} -{\widehat{\gamma}{}_{E}^{C}}'(t_{i+1}) \ \widehat{\sigma}_S^E(t_{Q+1-i})
                =S_1+S_2<\frac{\varepsilon}{2}+\frac{\varepsilon}{2}=\varepsilon
            \end{align*}
           by using~\eqref{eq:sigmaconv_zweiterteilSumme} and~\eqref{eq:sigmaconv_ersterteilSumme}.

            Altogether we have shown that for any $\varepsilon>0$ there exists a $Q(\varepsilon)$ with $ x_Q < \varepsilon$ for all $Q > Q(\varepsilon)$, i.e., it holds that $x_Q \to 0$ for $Q\to \infty$.

           It follows that
           \begin{align}
               \lim_{n\rightarrow\infty} \widehat{\sigma}_E^C(t_{n+1})=\Delta t\,x_n=0
           \end{align}
           and that the sequence $(\widehat{\sigma}_E^C(t_\AW{n}))_{n \in \mathbb{N}_0}$ is also bounded. 

            Analogously, it can be proven that all other transitions are bounded and converge to zero.

        \item\label{enum:proof3} The statement on $\widehat{\AW{\lambda}}$ can be proven analogously to the previous item~\ref{enum:proof2}. To do this, we make the following observations.
        \begin{itemize}
            \item With the assumptions~\eqref{eq:assumption_xiC} and~\eqref{eq:assumption_xiI}, we obtain the analogous statement to~\eqref{eq:sigmaconv_assumption_gamma}.
            \item From item~\ref{enum:proof2}, above, we know that 
            \begin{align*}
                \lim_{i\to\infty} \widehat{\sigma}_E^C(t_i)=0,
            \end{align*}
            \begin{align*}
                \lim_{i\to\infty} \widehat{\sigma}_C^I(t_i)=0
            \end{align*}
            and that these sequences are bounded.
        \end{itemize}
         With~\cref{lem:not_all_individuals_die} we can deduce that \AW{$\lim_{i\to \infty} (N- \widehat{D}(t_{i}))^{-1}$} is bounded. By assumption, the parameter $\phi(t)$ \AW{is} bounded. 
        Hence, we can apply the proof structure from before to the force of infection term $\widehat{\AW{\lambda}}$ to obtain the statement.       
        \item\label{enum:proof4} Since all transitions are nonnegative, the discretizations for the compartments $R$ and $D$ in~\eqref{eq:diskret_update} indicate that the sequences $(\widehat{R}(t_n))_{n \in \mathbb{N}_0}$ and $(\widehat{D}(t_n))_{n \in \mathbb{N}_0}$ are non\-de\-crea\-sing. By~\cref{thm:positivity_boundedness}, we already know that both sequences are bounded from above, which implies convergence. 
    \end{enumerate}
\end{proof}

\section{Numerical results}\label{sec:num}

In this section, we will first demonstrate the numerical convergence order of our nonstandard discretization scheme. Then, we show how \AW{a} standard ODE and our newly introduced IDE model behave at change points and for a COVID-19 inspired scenario. The presented IDE model and numerical experiments were implemented \AW{in C++} as a part of our high performance modular epidemics simulation software MEmilio~\cite{memilio121}.

\subsection{Order of convergence}
We start with examining the convergence order of the discretization method that we derived in~\cref{sec:discretization}. 

In order to determine the order of convergence, we can use the property that the continuous IDE model reduces to an ODE model under a special choice of parameters, see also~\ref{app:ide_ode}. Our parameter selection can be found in~\ref{app:parameters_IDEODE}. We use the solution of the corresponding ODE model solved with high precision as ground truth. In particular, we solve the ODE model using a Runge-Kutta scheme \AW{of fifth order} with a fixed time step of $\Delta t = 10^{-6}$ for $t \in [0, 70]$ and denote the solution (either compartment or flow size) ${u}^*_{\text{ODE}}$. The IDE model is solved with different time steps $\Delta t$ and the results are compared with those of the ODE model to calculate the convergence order of the numerical scheme.
To initialize the IDE model, we need transitions from the past. For this, we compute the respective transitions for $t \in [0,35]$ based on the results from the ODE simulation and use these for the initialization of the IDE model to make both models comparable. The simulation results of the IDE models and the ODE model are compared on the time interval starting at $t=35$ until $t_{\text{max}} = 70$. 
To compare the results, we calculate the relative error in $\|\cdot\|_2$-norm, i.e.,
\begin{align*}
   err_{\text{rel}}:= \frac{\left\lVert \widehat{u}_{\text{IDE}}-{u}^*_{\text{ODE}}\right\rVert_2}{\left\lVert{u}^*_{\text{ODE}}\right\rVert_{2}}
\end{align*} for each flow $\widehat{u}_{\text{IDE}}\in\{\widehat{\sigma}_{z_1}^{z_2} \text{ for appropriate }z_1,z_2\in\mathcal{Z}\}$ and for each compartment $\widehat {u}_{\text{IDE}}\in\widehat{\mathcal{Z}}$ obtained with the IDE model and where ${u}^*_{\text{ODE}}$ represents the corresponding solution of the ODE model. Here, we use a discrete $L^2$-norm that is defined by
\begin{align*}
    \left\lVert {u}\right\rVert_2=\left(\Delta t \sum_{i}u^2(t_i)\right)^{\frac{1}{2}}
\end{align*} and we evaluate the simulation results at the time points $t_i\in[35,70]$.
The results are depicted in~\cref{fig:convergence_l2}. We observe linear convergence for all transitions and all compartments. This is consistent with the results of
Messina et al.~\cite{messina_non-standard_2022} who derived a convergence order of one for compartment $S$ and their nonstandard numerical scheme.

\begin{figure}[H]
\begin{subfigure}{.5\textwidth}
    \centering
    \includegraphics[scale=0.45]{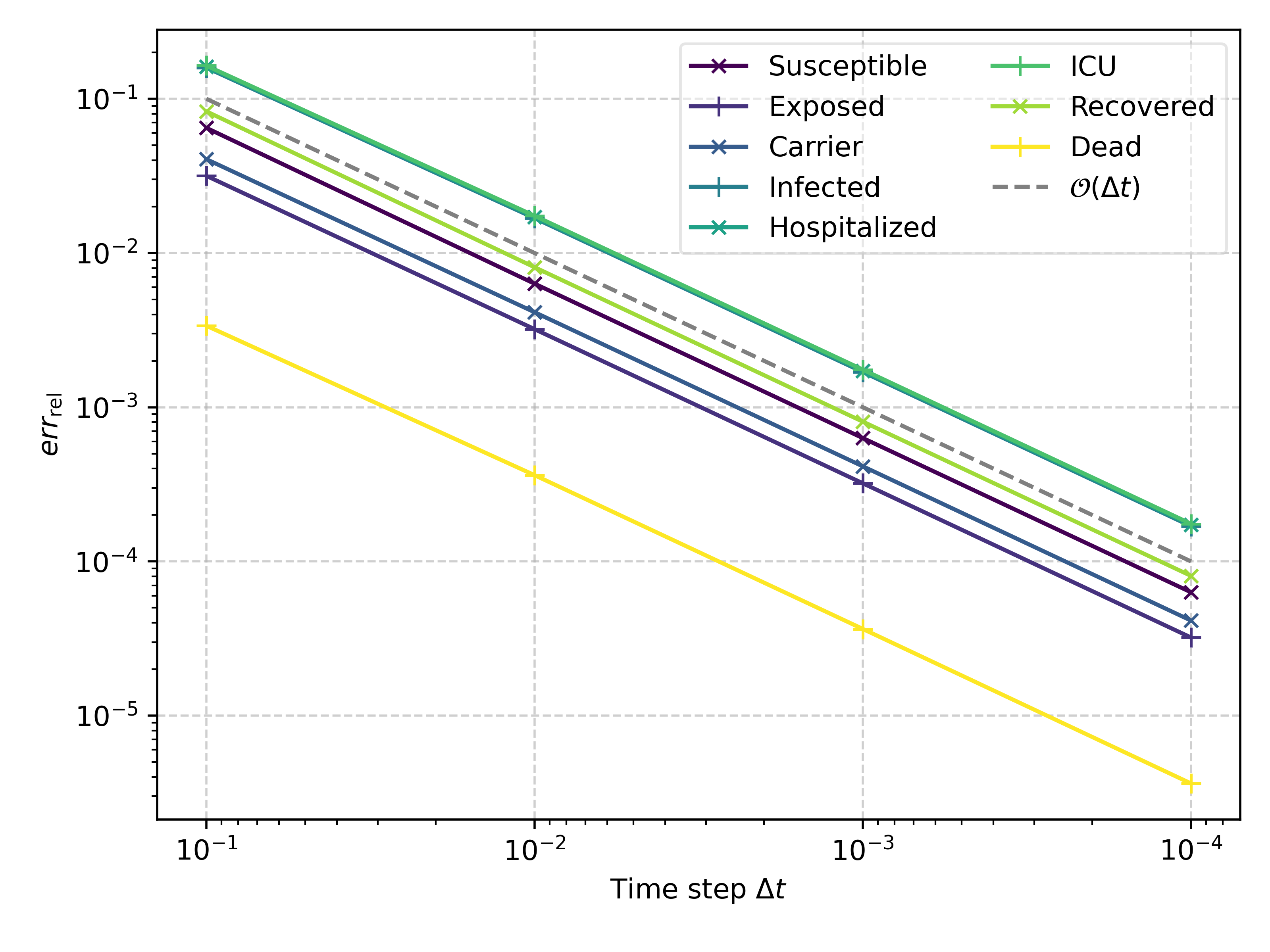}
\end{subfigure}
\begin{subfigure}{.5\textwidth}
    \centering
    \includegraphics[scale=0.45]{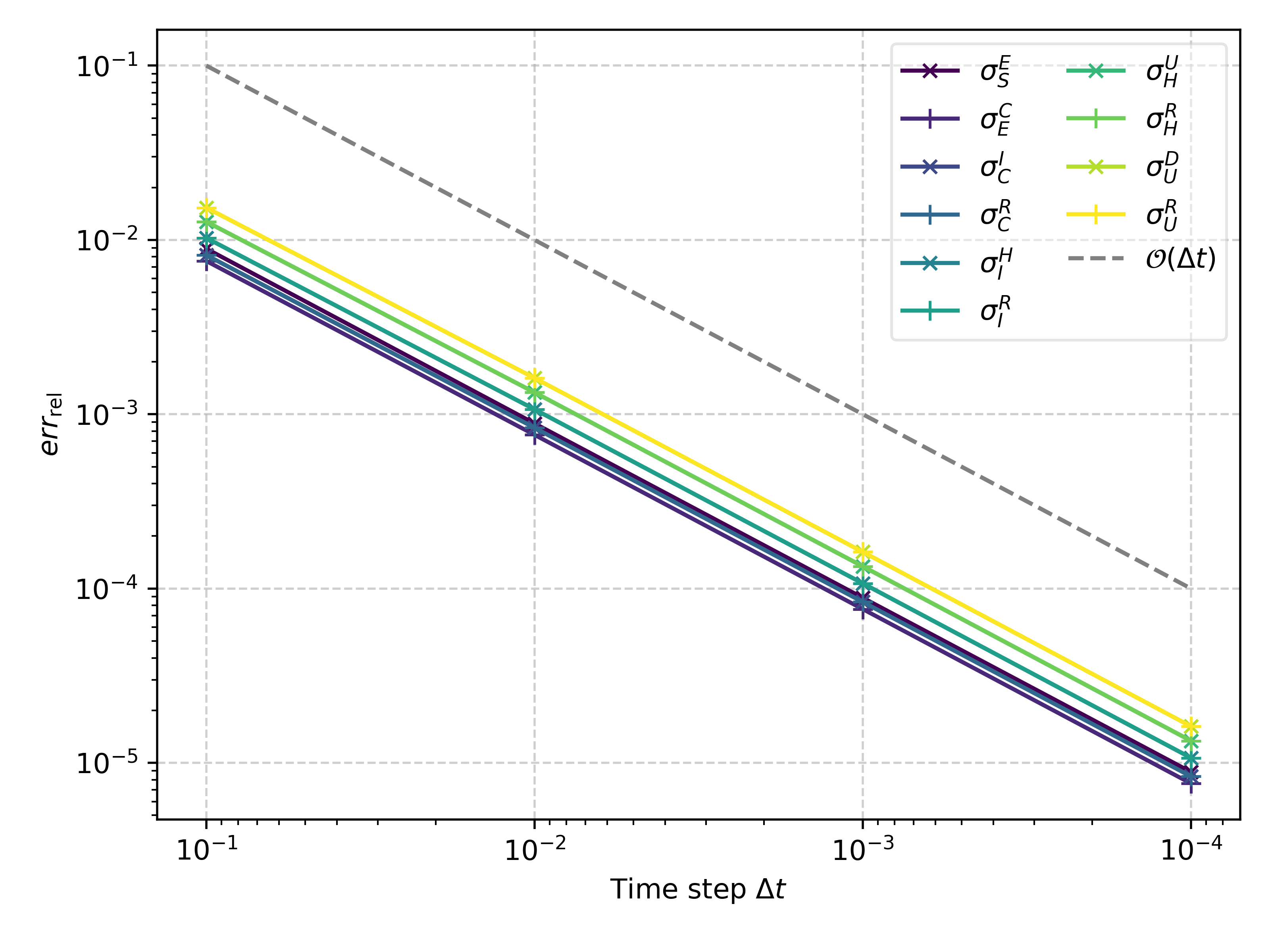}
\end{subfigure}
\caption{\textbf{Convergence of IDE model for compartments (left) and transitions (right).} Relative error of the simulation results of the IDE model with different time step sizes compared with numerical results of the ODE model with a step size $\Delta t=10^{-6}$ (ground truth) for the compartments (left) and the transitions (right). A linear function is plotted to compare the slope of the errors. }\label{fig:convergence_l2}
\end{figure}

\subsection{Model behavior at change points}\label{sec:changepoints}
This section presents a comparison of the behavior at change points of the IDE and ODE models. Change points may occur in response to the implementation or lifting of nonpharmaceutical interventions.
To analyze these dynamics, we use epidemiological parameters that are realistic for SARS-CoV-2. The parameters for the ODE and the IDE model are described in~\ref{app:parameters_SARS}, including the stay time distributions used for the IDE model. The parameters are inferred using the results of~\cite{kerr_covasim_2021} and~\cite{kuhn_assessment_2021}. Both models are initialized such that the number of new transmissions is constant at the beginning of the simulation. The contact rate is set accordingly. To model the implementation or lifting of nonpharmaceutical interventions, the contact rate is halved or doubled, respectively, in both models after two days. The simulation results for a simulation period of $12$ days are shown in~\cref{fig:changepoints} in form of the daily new transmissions. The IDE model is solved with step size $\Delta t= 10^{-2}$ and the ODE model with a Runge-Kutta scheme with the same fixed step size.

Firstly, we note that, as expected, the change in the contact rate is immediately observable in the new transmissions and a doubling/halving of the daily new transmissions occurs in both simulation results.
Afterward, we note that the ODE model responds more quickly to changes in the contact rate, while the predicted new transmissions in the IDE model remain constant for a period of time. This lag time observed for the IDE model is a realistic phenomenon. For example, Dey et al.~\cite{dey_lag_2021} or Guglielmi et al.~\cite{guglielmi_identification_2023} found a nontrivial lag time between the implementation of a nonpharmaceutical intervention and the change in the case data. In contrast to the ODE model, the IDE model naturally incorporates this time delay.
After the delay, the slope appears to be steeper in the IDE than in the ODE model. In summary, it can be stated that the different assumptions regarding the stay time distributions result in notable differences in the behavior at change points.

\begin{figure}[ht]
\begin{subfigure}{.5\textwidth}
    \centering
    \includegraphics[scale=0.45]{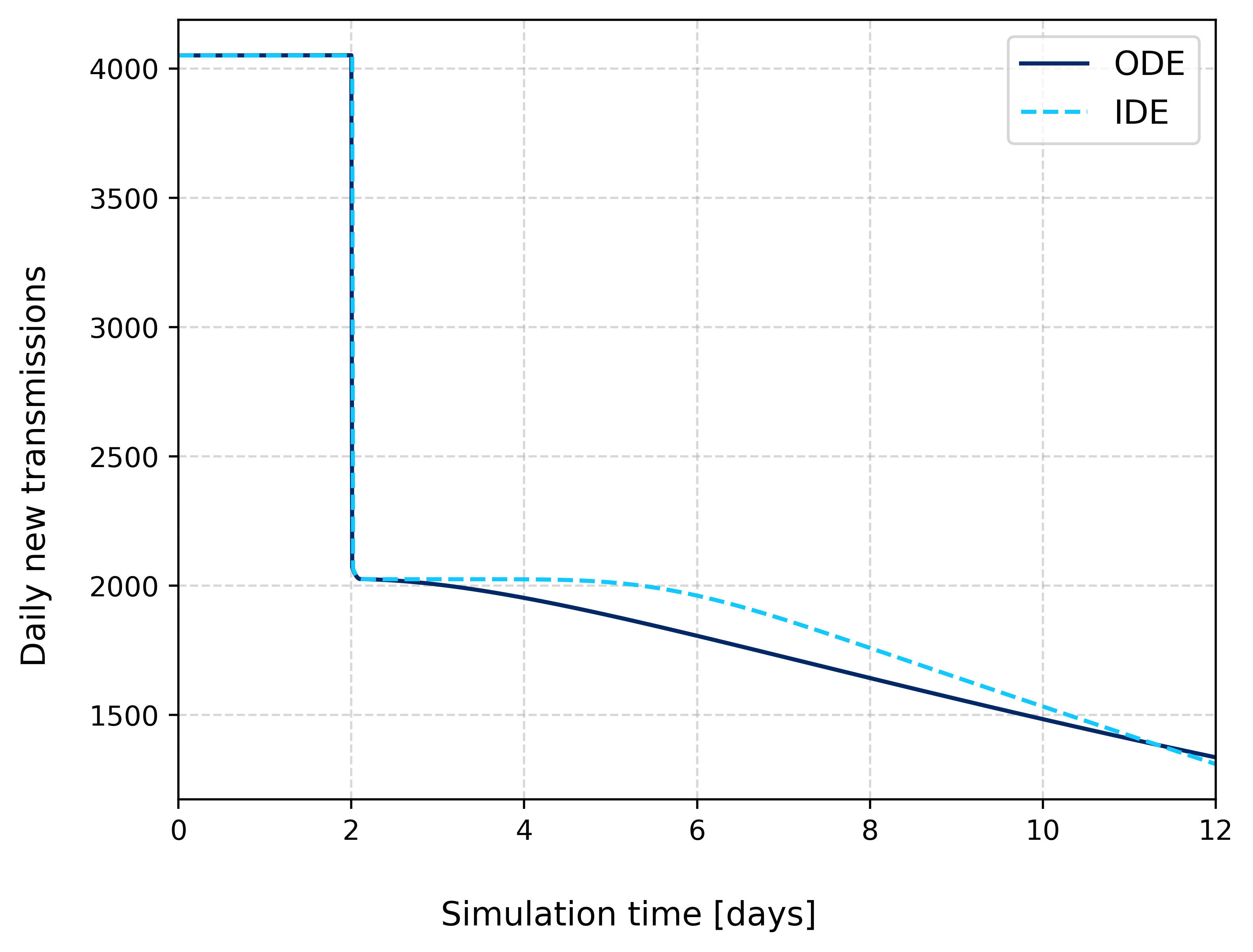}
\end{subfigure}
\begin{subfigure}{.5\textwidth}
    \centering
    \includegraphics[scale=0.45]{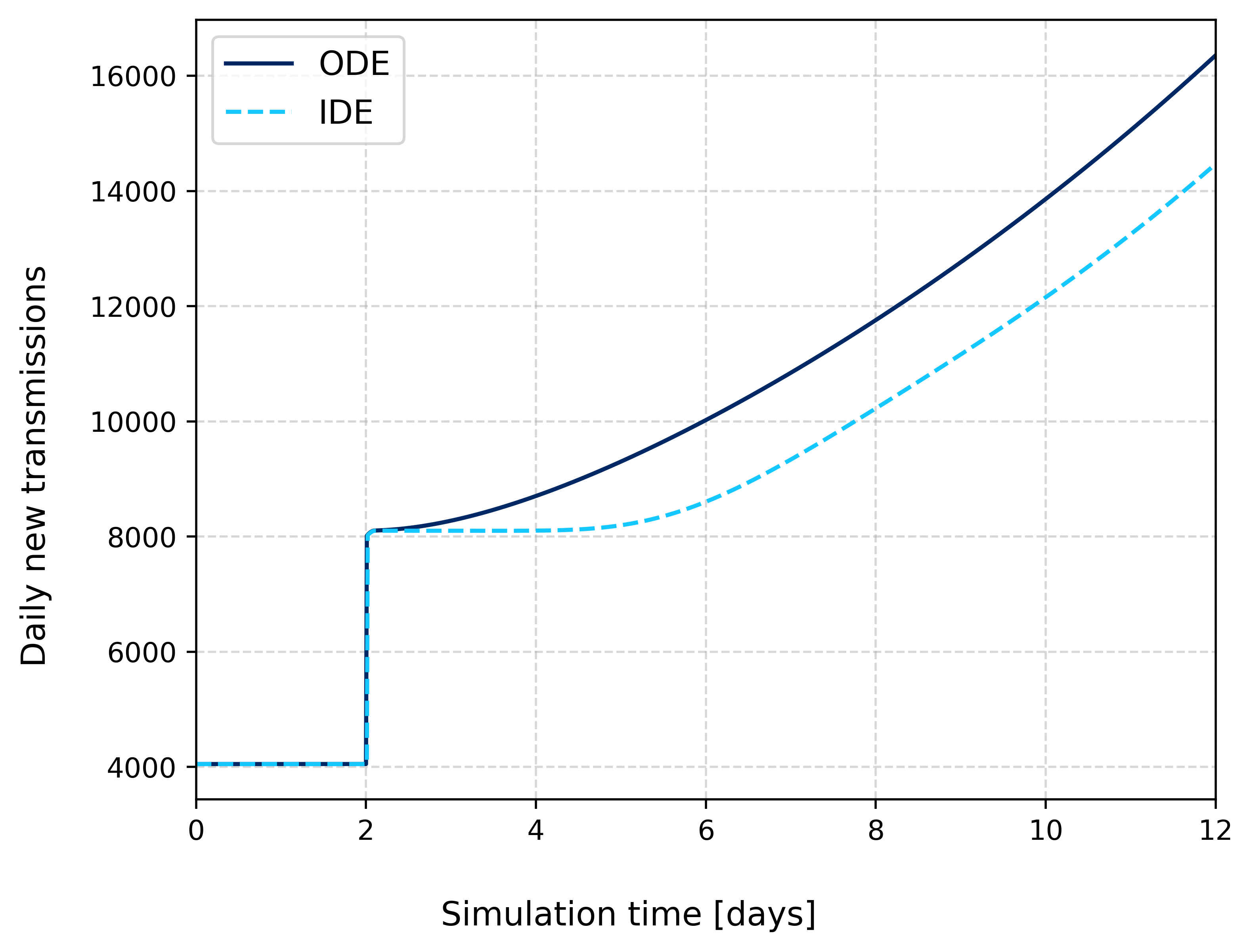}
\end{subfigure}
\caption{\textbf{Daily new transmissions at change points.} Comparison of the simulation results for the daily new transmissions
of the ODE and the IDE model for a halving (left) or doubling (right) of the contact rate $\phi(t)$ after two simulation days.}
\label{fig:changepoints}
\end{figure}

\subsection{A COVID-19 inspired real scenario}\label{sec:covid_scenario}

In this section, we demonstrate that the newly introduced IDE model can be used to predict realistic infection dynamics. To do this, we are looking at the spread of the COVID-19 disease in Germany in October $2020$. The simulation results of the IDE model are compared with interpolated reported data as well as with simulation results obtained using an appropriate ODE model. The reported data is processed in the following way. For simplicity, we assume in the explanation below that $\Delta t =1$, i.e., we assume that the considered time points correspond to days. In the implementation, a smaller step size is used and the data is interpolated correspondingly.

As in~\cref{sec:initialization}, we assume that the reported cases correspond to the mildly symptomatic individuals. In~\cite{RKI_data_2023}, confirmed cases are reported cumulatively. We denote them by $\Sigma_{\text{I,RKI}}$. We can infer $\widetilde{\sigma}_C^I$ from the reported data by
\begin{align*}
    \widetilde{\sigma}_C^I(t_n) = \Sigma_{\text{I,RKI}}(t_n) - \Sigma_{\text{I,RKI}}(t_{n}-1).
\end{align*}
With this we can initialize the flows of the IDE model as described in~\cref{sec:initialization}.

Now, we explain how we use the reported data to obtain the number of new transmissions, mildly symptomatic individuals as well as deaths. We will compare our simulation results to these values. 

The number of new transmissions at time $t_n$ is obtained by
\begin{align*}
    \sigma_{S,\text{rep}}^E(t_n) = \frac{1}{\mu_C^I}\left(\Sigma_{\text{I,RKI}}(t_n +  T_C^I + T_E^C) - \Sigma_{\text{I,RKI}}(t_n +  T_C^I + T_E^C -1)\right),
\end{align*}
where we assume that it takes individuals exactly $T_E^C+T_C^I$ days from getting infected to developing symptoms. The reported data is interpolated linearly in between two time steps if necessary.

To compute the number of individuals in compartment $I$, we assume that a share of $\mu_I^H$ individuals stays exactly $T_I^H$ days in compartment $I$ and the remaining share of $1-\mu_I^H$ stays exactly $T_I^R$ days in compartment $I$. This leads to
\begin{align*}
    I_{\text{rep}}(t_n) = \mu_I^H\left(\Sigma_{\text{I,RKI}}(t_n) - \Sigma_{I,\text{RKI}}(t_n- T_I^H)\right)+  \left(1-\mu_I^H\right)\left(\Sigma_{\text{I,RKI}}(t_n) - \Sigma_{I,\text{RKI}}(t_n- T_I^R)\right).
\end{align*}
Deaths are reported with the date when the infection is assumed to have taken place. To extrapolate the date of death, we shift the reported data (denoted by $D_{\text{RKI}}$) according to the mean stay times, i.e.,
\begin{align*}
    D_{\text{rep}}(t_n) = D_{\text{RKI}}(t_n - T_I^H - T_H^U - T_U^D),
\end{align*}
where $D_{\text{RKI}}(t)$ is the subset of deceased patients from $\Sigma_{\text{I,RKI}}(t)$. Once more, we interpolate linearly when required. The above extrapolation is also used to infer the number of deaths at the simulation start needed for the initialization of the IDE model, cf.~\cref{sec:initialization}.
To compare the number of patients in intensive care units to real data, we use reported intensive care patients directly~\cite{divi2024}.

The parameters described in~\ref{app:parameters_SARS} are again utilized. We use again a step size of $\Delta t= 10^{-2}$ for the IDE model and solve the ODE model with a fifth order Runge-Kutta scheme with constant step size $\Delta t$. 
The IDE model is initialized using \AW{reported} data as described in~\cref{sec:initialization}. The initial compartment sizes of the ODE model are set to the sizes calculated using the IDE model to make both models comparable. We choose Oct 1, 2020 as the start date and simulate for $45$ days, as this period was significant in the pandemic and the data situation on the number of cases comparatively good. As it is hard to measure the contact rate accurately in every time period, the contact rate $\phi(t)$ is set such that the daily new transmissions of the simulation results are consistent to the reported data at the beginning of the simulation.

We use the same contact rate for both model types and introduce a contact reduction at the time when a drop in the reported data can also be detected; for details see the appendix. 
This factor is determined on the basis of~\cite{kuhn_assessment_2021} and corresponds approximately to the strength of the nonpharmaceutical measures applied there for the fall.

The simulation results obtained for the described scenario are depicted in~\cref{fig:real_scenario}. The reported data are very well predicted by the IDE model regarding daily new transmissions, number of mildly infected individuals as well as ICU patients. The ODE model tends to underestimate the \AW{reported} data. Our proposed IDE model seems to be well suited for modeling realistic scenarios. Only the number of deaths appears to be underestimated by the IDE model as well. This could be explained through the fact that in our model, we only allowed \AW{individuals} to die in intensive care units, see also~\cref{fig:IDESECIR} or that, in fact, mild cases might have gone unseen and that \AW{matching the simulation results to} reported data underestimates upcoming deaths.

\begin{figure}[htb]
\begin{subfigure}{.5\textwidth}
    \centering
    \includegraphics[scale=0.4]{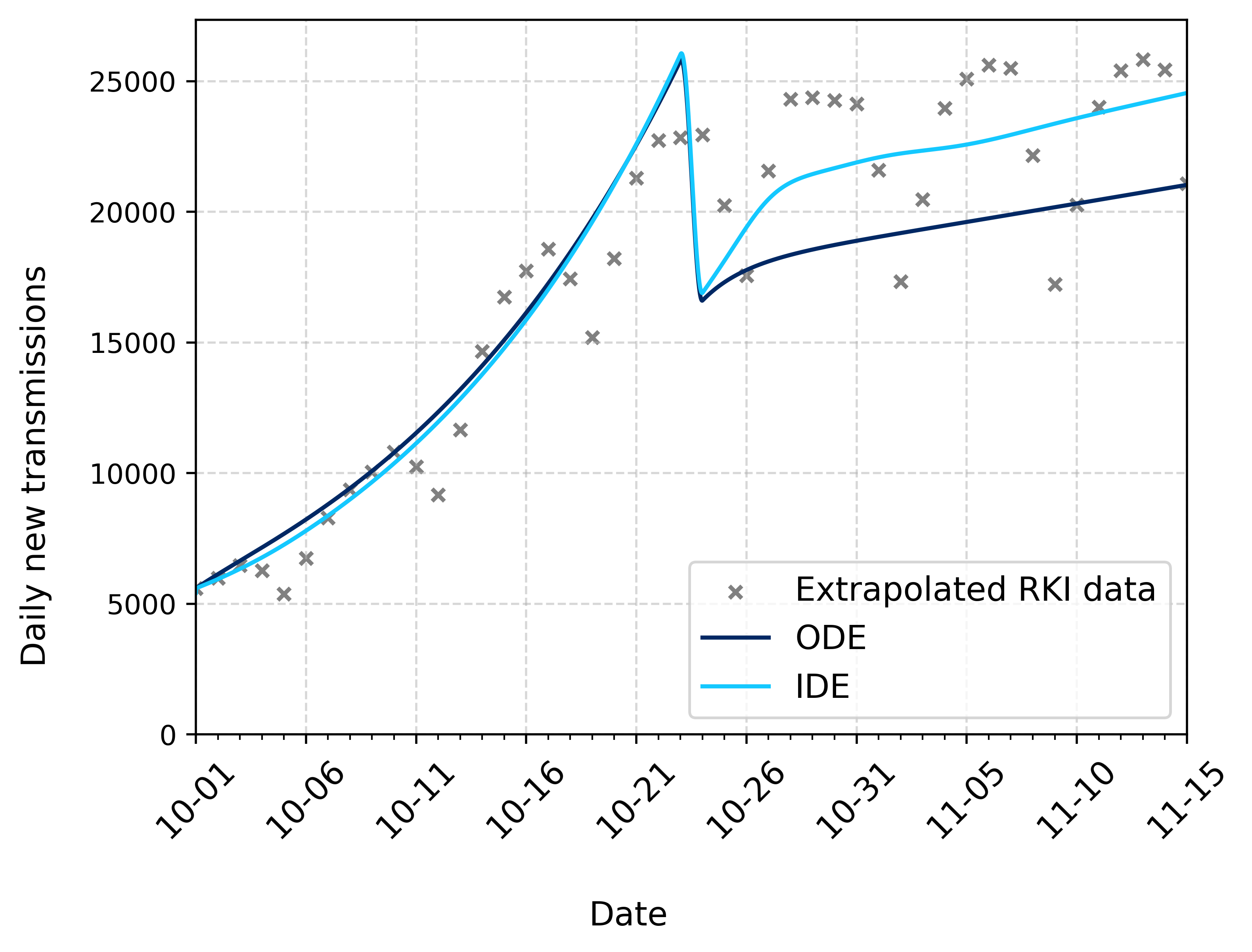}
\end{subfigure}
\begin{subfigure}{.5\textwidth}
    \centering
    \includegraphics[scale=0.4]{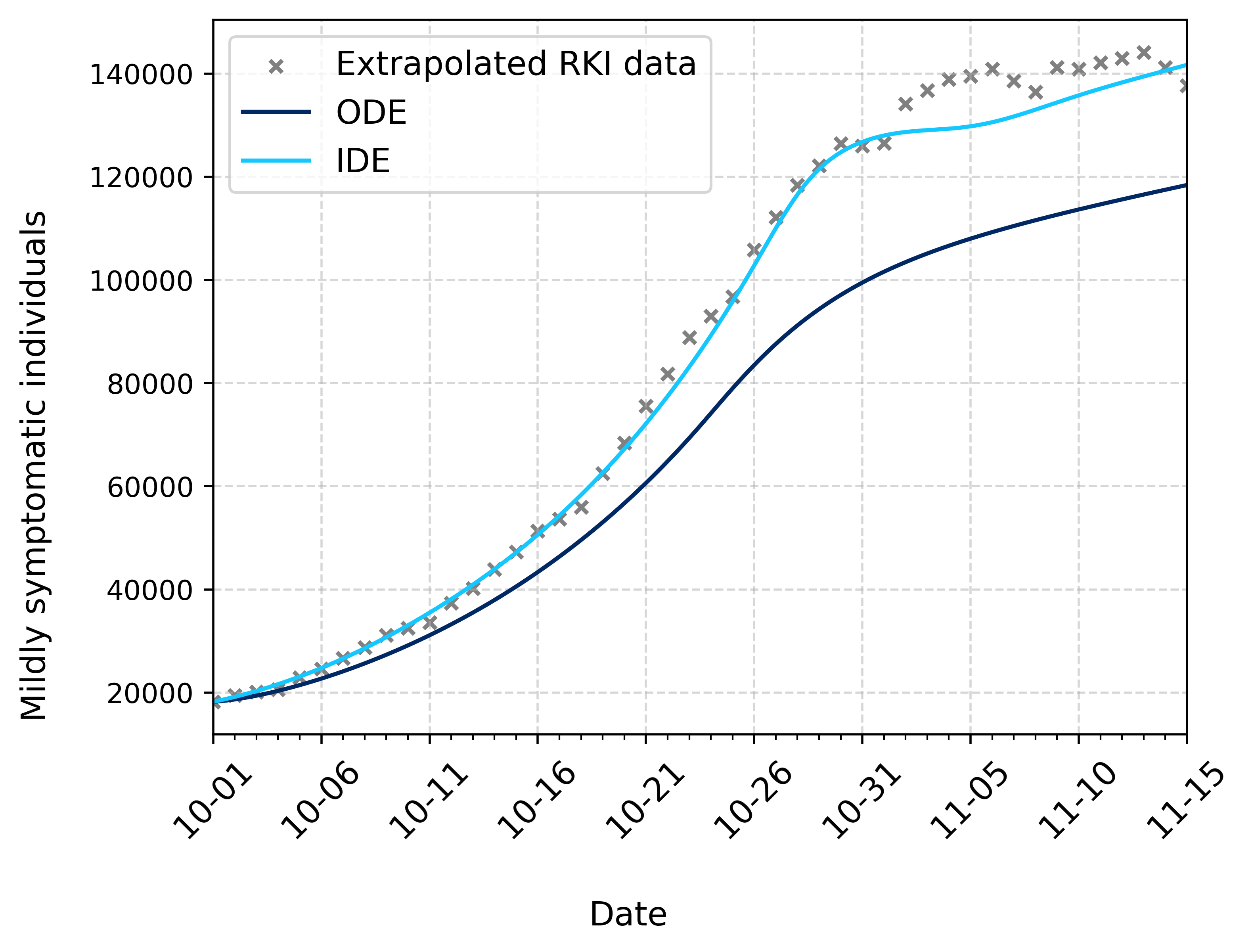}
\end{subfigure}
\begin{subfigure}{1.\textwidth}
    \begin{subfigure}{.5\textwidth}
        \centering
        \includegraphics[scale=0.4]{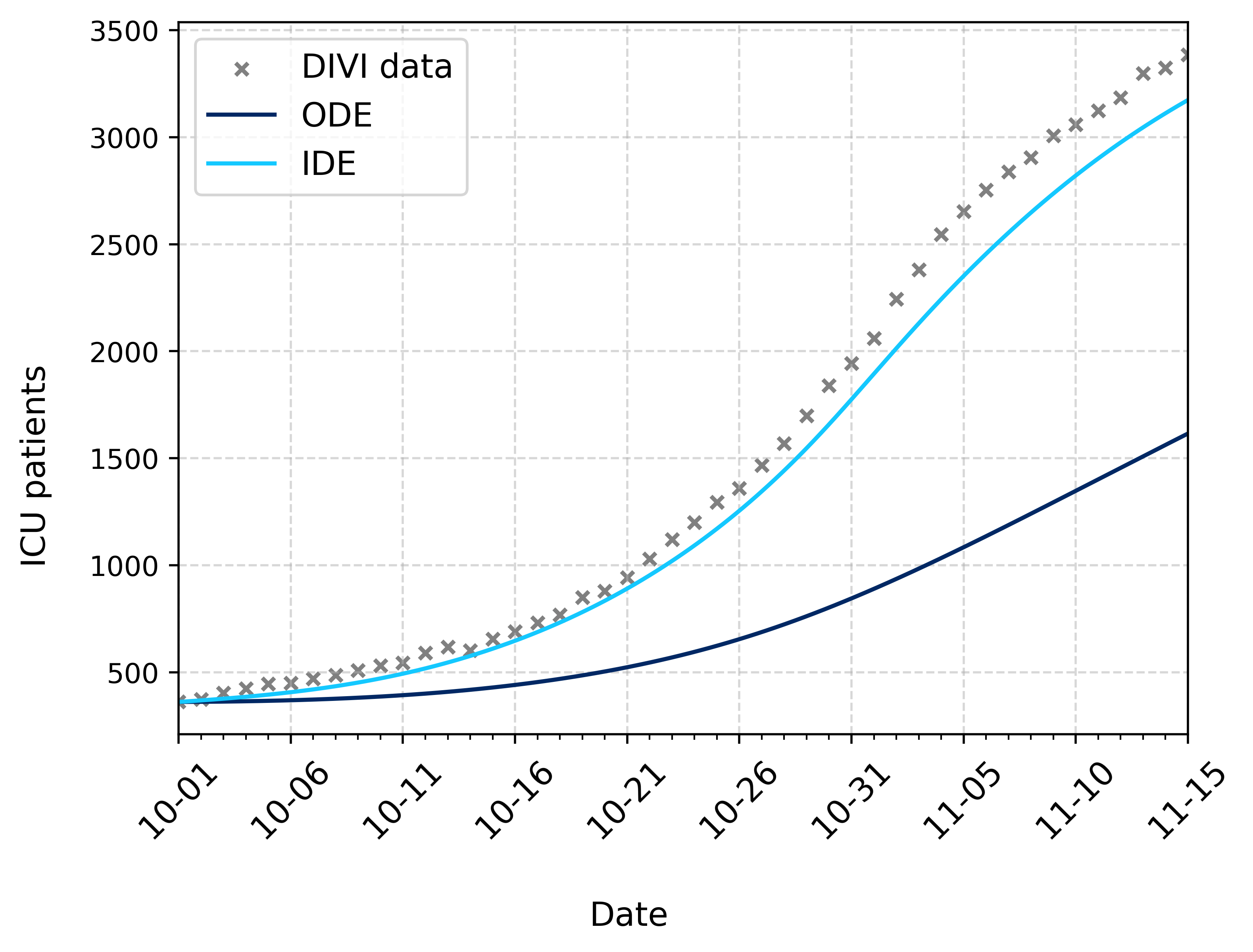}
    \end{subfigure}
    \begin{subfigure}{.5\textwidth}
        \centering
        \includegraphics[scale=0.4]{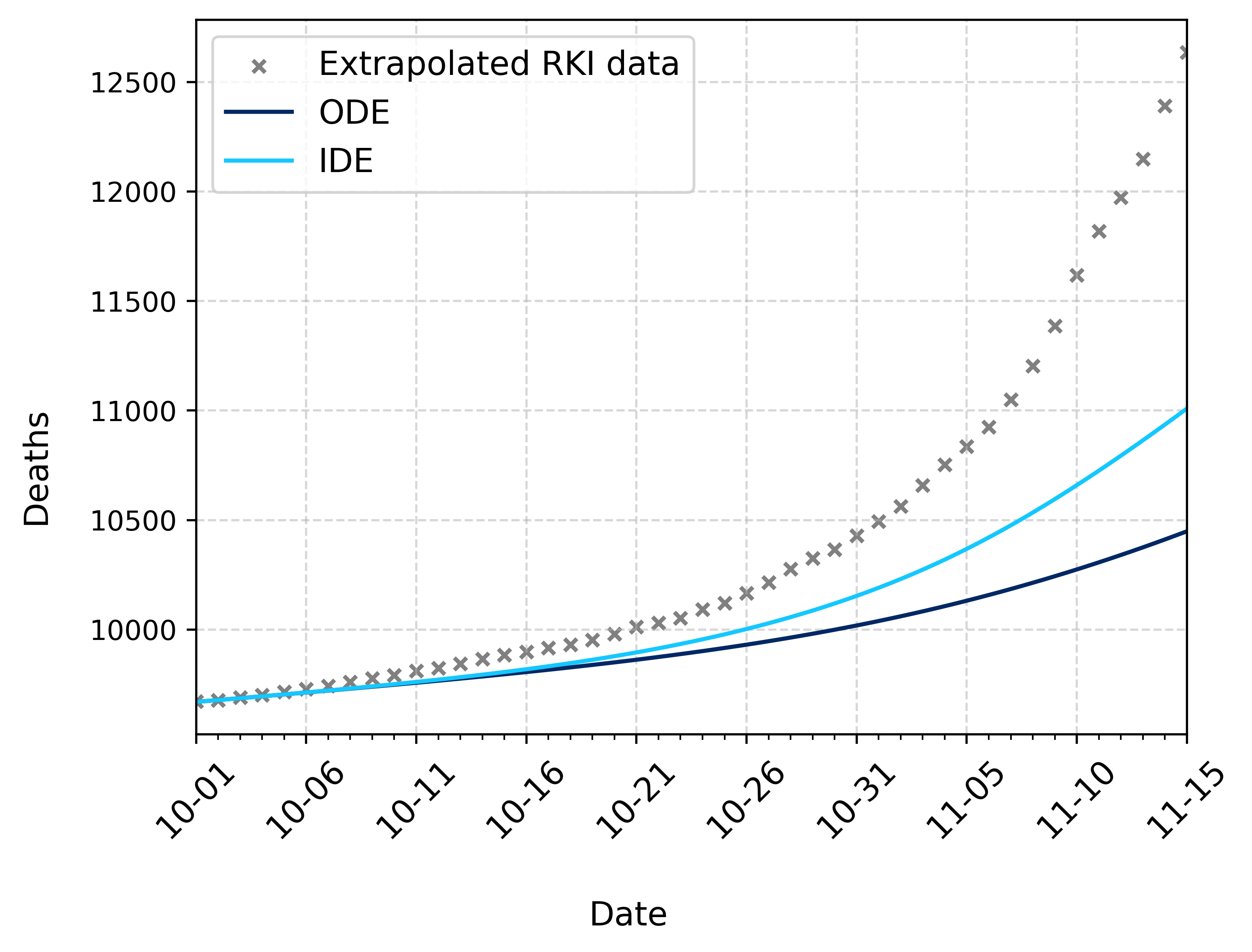}
    \end{subfigure}
\end{subfigure}
\caption{\textbf{Simulation results for COVID-19 in Germany from Oct 1, 2020, onwards.} Comparison of extrapolated real data with the simulation results of the IDE and the ODE model. The results are shown for the number of daily new transmissions (top left), the number of mildly symptomatic individuals (top right), the number of patients in intensive care units (bottom left) and the number of deaths (bottom right).}
\label{fig:real_scenario}
\end{figure}

\section{Discussion}\label{sec:discussion}

Models for infectious disease dynamics always use implicit or explicit assumptions on the dynamic process. As exponential stay times have been found to be unrealistic by different authors~\cite{wearing_appropriate_2005,donofrio_mixed_2004, faesTimeSymptomOnset2020, challenMetaanalysisSevereAcute2022, madewellSerialIntervalIncubation2023}, we proposed a model based on integro-differential equations that allows for modeling of more realistic stay time distributions. However, for many diseases the true stay time distributions are unknown or, at least in the beginning of a pandemic, difficult to get good estimations of. While this poses a heavy limitation on the advancement given here, in the worst case, our model reduces to the mostly used ODE model forms. However, in the COVID-19 scenario we found that the results based on data published rather early in the pandemic, led to good results of the IDE model without the need of extensive fitting. 
The deviations in deaths could be explained by a certain number of undetected infections, to which we have not fitted the model. Let us note that, by default, we would also expect our model to constantly overestimate reported ICU numbers as deaths in our model are only possible from ICU.

Furthermore, epidemiological parameters are highly dependent on age and mobility processes can become drivers of the spread, see e.g.~\cite{kuhn_assessment_2021}. 
In our simulations we averaged over age-dependent parameters by using the respective proportion of the age group compared to the total population to simulate with a model without age resolution. In~\cref{fig:proportions_per_agegroup}, one can see that the age distribution of the (confirmed) cases heavily depends on the time period under consideration. Hence, the appropriate average value had to be fitted differently for different simulation periods. To further enhance model accuracy, age groups need to be included into the model. Implementing spatial resolution allows for modeling the impact of local implementations of nonpharmaceutical interventions and mobility across counties.
These aspects will be the subject of further research. 

\begin{figure}
    \centering
    \includegraphics[scale=0.4]{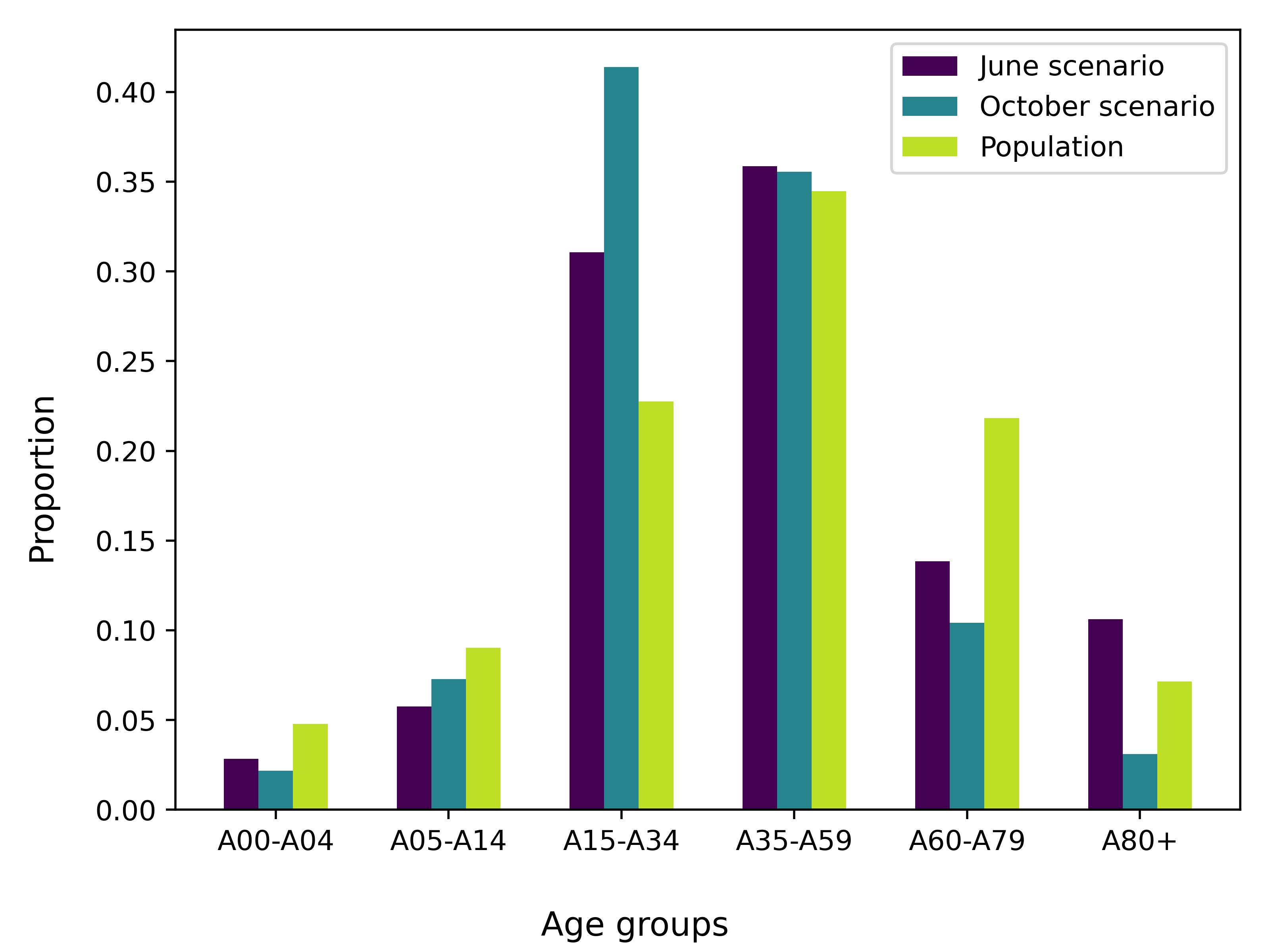}
\caption{\textbf{Share of considered population in respective age groups.} For dates Jun 1, 2020, and Oct 1, 2020, we consider the number of confirmed cases between $T_I^H + T_H^U + T_U$ and $T_I^H + T_H^U$ days before these dates. This corresponds to the individuals that we expect to be in compartment $U$ at time $t_0$. We use data on the confirmed cases of COVID-19 in Germany provided by~\cite{RKI_data_2023}. Above, the proportion per age group of these confirmed cases is depicted. For comparison, the proportion of each age group in the total population in Germany as reported in~\cite{regionaldatenbank_deutschland_fortschreibung} is shown.\label{fig:proportions_per_agegroup}}
\end{figure}

\section{Conclusion}\label{sec:conclusion}

In this paper, we introduced a detailed model based on integro-differential equations (IDE) that allows the consideration of arbitrary stay time distributions. The proposed model is a generalization of an ODE model, which is restricted to exponentially distributed stay times.
We further extended a nonstandard numerical scheme to solve the resulting equations of the IDE model. We provided theoretical results for the numerical solution scheme, proving that important biological properties regarding, e.g., positivity and boundedness, are preserved. Furthermore, we examined the behavior in the limit of the numerical solution and provided theoretical results. Numerically, we demonstrated a linear convergence rate for all considered solution elements -- which is consistent with the results of~\cite{messina_non-standard_2022} for the Susceptible compartment. 

To demonstrate the importance of choosing appropriate stay time distributions when modeling epidemic outbreaks, we compared our proposed IDE model with a corresponding ODE model. The model behavior is compared at change points and in a COVID-19 inspired real scenario. We observed significant differences in both cases. In contrast to the ODE model, the IDE model shows a realistic delay time after changing the contact rate. This confirms findings of the literature, where a delay between the implementations of nonpharmaceutical interventions and their effect on the case data was observed. 
In the COVID-19 scenario, we found that the results based on the IDE model lead to more accurate predictions than the ODE model. This shows the benefit of using realistic stay time distributions compared to restricting oneself to exponential stay time distributions. 

We are optimistic that our modeling approach as well as the nonstandard numerical solution scheme to preserve important mathematical-biological properties can easily be adapted to other infectious diseases such as, e.g., Influenza. 

\section*{Acknowledgements}
\noindent AW, LP, and MJK have received funding from the Initiative and Networking Fund of the Helmholtz Association (grant agreement number KA1-Co-08, Project LOKI-Pandemics). The authors LP and MJK have received funding by the German Federal Ministry of Education and Research under grant agreement 031L0297B (Project INSIDe). LP, HT, and MJK have received funding by the German Federal Ministry for Digital and Transport under grant agreement FKZ19F2211A (Project PANDEMOS).

\section*{Competing interests}
\noindent The authors declare to not have any competing interests.

\section*{Data availability}
\noindent The MEmilio repository is publicly available under \url{https://github.com/SciCompMod/memilio}. All model functionality is available with MEmilio v1.2.1~\url{https://zenodo.org/records/13341171}. \AW{ A detailed description on how to reproduce the simulation results presented in this paper, in addition to the complete set of plot files, is accessible at~\url{https://github.com/SciCompMod/memilio-simulations}.}

\section*{Author Contributions}

    \noindent\textbf{Conceptualization:} Anna Wendler, Martin Kühn\\
    \noindent\textbf{Data Curation:} Anna Wendler, Lena Plötzke\\
    \noindent\textbf{Formal Analysis:} Anna Wendler, Lena Plötzke, Hannah Tritzschak, Martin Kühn\\
    \noindent\textbf{Funding Acquisition:} Martin Kühn\\
    \noindent\textbf{Investigation:} Anna Wendler, Lena Plötzke, Hannah Tritzschak, Martin Kühn\\
    \noindent\textbf{Methodology:} Anna Wendler, Lena Plötzke, Martin Kühn\\
    \noindent\textbf{Project Administration:} Martin Kühn\\
    \noindent\textbf{Resources:} Martin Kühn\\
    \noindent\textbf{Software:} Anna Wendler, Lena Plötzke, Hannah Tritzschak\\
    \noindent\textbf{Supervision:} Martin Kühn \\
    \noindent\textbf{Validation:} All authors \\
    \noindent\textbf{Visualization:} Anna Wendler, Lena Plötzke\\
    \noindent\textbf{Writing – Original Draft:} Anna Wendler, Lena Plötzke\\
    \noindent\textbf{Writing – Review \& Editing:} All authors

\bibliographystyle{elsarticle-num}

\newpage
\pagenumbering{roman}
\appendix
\FloatBarrier

\section{Relation between IDE and ODE model}\label{app:ide_ode}

In this section, we show that our proposed model using integro-differential equations is indeed a generalization of the ODE model as proposed in~\cite{kuhn_assessment_2021} if stay times of the ODE model were only dependent on the starting compartment, i.e., for instance, $T_C=T_C^I=T_C^R$. In the following, we show that using exponential distributions to describe the transitions between compartments in the IDE model yields the corresponding simplification to the ODE model.

The ODE-SECIR model of~\cite{kuhn_assessment_2021} without age resolution is given by
\begingroup
\allowdisplaybreaks[3]
\begin{align} \label{eq:ode_secir}
\begin{aligned}
   S'(t) &= 
    - \frac{S(t)}{N-D(t)}\, \phi(t)\,\AW{\rho}\, \Big( \AW{\xi_C}\, C(t)+ \AW{\xi_I}\, I(t)\Big),\\ 
    E'(t) &= 
    \frac{S(t)}{N-D(t)}\, \phi(t)\,\AW{\rho}\,\Big( \AW{\xi_C}\, C(t)+ \AW{\xi_I}\, I(t)\Big) 
    - \frac{1}{T_E}\, E(t),\\ 
    C'(t) &= 
    \frac{1}{T_{E}}\,E(t)  
    - \frac{1}{T_{C}}\, C(t),\\
    I'(t) &= 
    \frac{{\mu_{C}^{I}}}{T_{C}} \,C(t) 
    -\frac{1}{T_{I}}\, I(t) ,\\    
    H'(t) &=
     \frac{\mu_{I}^{H}}{T_{I}}\, I(t)
     - \frac{1}{T_{H}}\, H(t),\\
     U'(t) &= 
     \frac{\mu_{H}^{U}}{T_{H}}\, H(t)
     - \frac{1}{T_{U}} \,U(t) ,\\
     R'(t)&= 
     \frac{{1-\mu_{C}^{I}}}{T_{C}}\, C(t)
     + \frac{1 - \mu_{I}^{H}}{T_{I}}\,I(t)
     + \frac{1-\mu_{H}^{U}}{T_{H}}\, H(t)
     + \frac{1 - \mu_{U}^{D}}{T_{U}}\, U(t) ,\\
     D'(t) &=
     \frac{\mu_{U}^{D}}{T_{U}}\, U(t).
\end{aligned}
\end{align}
\endgroup
where $T_{z_1}$ is the average stay time in compartment $z_1\in\mathcal{Z}$ and the meaning of the other parameters is as before. Note that~\cite{kuhn_assessment_2021} allowed setting stay times dependent on start and destination compartment $z_1$ and $z_2$. However, if $T_{z_1}^{z_2}$ and $T_{z_1}^{z_3}$ differ for $z_2\neq z_3$, then $\mu_{z_1}^{z_2}$ looses its meaning as probability since the Markov property of the system does not consider transition histories. The parameter $\xi_I$ has been denoted $\widetilde{\beta}$ in~\cite{kuhn_assessment_2021} and $\xi_C$ had only been mentioned in~\cite{koslow_appropriate_2022}, as it was set to one in the prior publications. 

To verify that the IDE model~\eqref{eq:IDESECIR} reduces to the ODE model~\eqref{eq:ode_secir} under a certain choice of parameters, we first set $\xi_C(\tau) = \xi_C$\AW{, }$\xi_I(\tau) = \xi_I$ \AW{and $\rho_C(\tau)=\rho_I(\tau)=\rho$} in the IDE model. For the transition distributions, we choose 
\begin{align}\label{eq:ide_ode_gamma}
    \gamma_{z_1}^{z_2}(\tau) = \exp(-\frac{1}{T_{z_1}}\, \tau),
\end{align}
regardless of $z_2$ and where $T_{z_1}$ is the mean stay time in compartment $z_1$ as introduced for the ODE model. Here,
\begin{align*}
    1-\gamma_{z_1}^{z_2}(\tau) = 1-\exp(-\frac{1}{T_{z_1}}\, \tau)
\end{align*} is indeed a CDF, i.e., the CDF function of the exponential distribution with parameter $(T_{z_1})^{-1}$. Hence, this choice for $\gamma_{z_1}^{z_2}$ is valid for the IDE model. Note that
\begin{align}\label{eq:ide_ode_gamma_ableitung}
    {\gamma_{z_1}^{z_2}}'(\tau) = -\frac{1}{T_{z_1}} \exp(-\frac{1}{T_{z_1}}\tau) = -\frac{1}{T_{z_1}} \, \gamma_{z_1}^{z_2}(\tau).
\end{align}
Based on these assumptions, we have, e.g., $\gamma_C^I = \gamma_C^R$, so that the stay time distribution in compartment $C$ is identical for all individuals. 
The same holds for the subsequent compartments.

We will now demonstrate that, under the specified parameterization of the IDE model, the aforementioned model is in fact equivalent to the ODE model.
The equation for $S'$ will serve as a starting point. With the chosen parameters, we obtain
\begin{align*}
    S'(t) &\overset{\eqref{eq:IDESECIR}}{=}  -\frac{S(t)}{N-D(t)}\, \phi(t)\,  
    \int_{-\infty}^t \xi_C(t-x)\,\AW{\rho_C(t-x)} \left(\mu_C^I\,\gamma_C^I(t-x)+\left(1-\mu_C^I\right)\gamma_C^R(t-x)\right) \sigma_E^C(x) \\
        & \hspace{3.1cm}  + \xi_I(t-x)\,\AW{\rho_I(t-x)} \left(\mu_I^H\,\gamma_I^H(t-x)+\left(1-\mu_I^H\right)\gamma_I^R(t-x)\right) \sigma_C^I(x) \, \d x
        \\
        &= -\frac{S(t)}{N-D(t)}\, \phi(t)\,  \rho\, \Big(\xi_C\int_{-\infty}^t  \left(\mu_C^I\,\gamma_C^I(t-x)+\left(1-\mu_C^I\right)\gamma_C^R(t-x)\right) \sigma_E^C(x) \, \d x \\
        &\hspace{2.6cm} +\xi_I\int_{-\infty}^t  \left(\mu_I^H\,\gamma_I^H(t-x)+\left(1-\mu_I^H\right)\gamma_I^R(t-x)\right) \sigma_C^I(x) \, \d x \Big)
       \\
       &\overset{\eqref{eq:IDESECIR}}{=}-\frac{S(t)}{N-D(t)}\, \phi(t)\, \rho\,\Big(\xi_C \,C(t)+\xi_I\, I(t)\Big).
\end{align*}
Hence, we get the equation for $S'$ of the ODE model once we show that the equations for compartments $C$ and $I$ of the IDE model coincide with those of the ODE model. 

We proceed by inserting our chosen transition distribution into the equation for $E$ of the IDE model and demonstrating that it reduces to the corresponding equation of the ODE model.
First, we make the observation that, using~\eqref{eq:ide_ode_gamma_ableitung}, the flow $\sigma_E^C(t)$ can be expressed as 
\begin{align}
\begin{aligned}\label{eq:ODE_sigma_EC}
    \sigma_E^C(t) &= - \int_{-\infty}^{t} {\gamma_{E}^{C}}'(t-x)\,\sigma_{S}^{E}(x)\ \d x =  \frac{1}{T_E} \, \int_{-\infty}^{t} {\gamma_{E}^{C}}(t-x)\,\sigma_{S}^{E}(x)\ \d x  \\
    &= \frac{1}{T_E} \, E(t).
\end{aligned}
\end{align}
Therefore, through~\eqref{eq:IDEAblKomp}, it holds that
\begin{align*}
    E'(t) &= \sigma_S^E(t) - \sigma_E^C(t)\\
    &=  - S'(t) - \frac{1}{T_E} \, E(t).
\end{align*}

For compartment $C$, we start again with an observation on the respective transitions $\sigma_C^I$ and $\sigma_C^R$. Note that inserting~\eqref{eq:ide_ode_gamma_ableitung} yields
\begin{align}
\begin{aligned}\label{eq:ODE_sigma_CI}
    \sigma_C^I(t) &=  - \int_{-\infty}^{t} {\gamma_{C}^{I}}'(t-x)\,\mu_{C}^{I}\,\sigma_{E}^{C}(x)\, \d x = \frac{\mu_C^I}{T_C} \int_{-\infty}^{t} {\gamma_{C}^{I}}(t-x)\,\sigma_{E}^{C}(x)\, \d x \\
    &= \frac{\mu_C^I}{T_C}\, C(t),
    \end{aligned}
    \end{align}
and
\begin{align}
\begin{aligned}\label{eq:ODE_sigma_CR}
    \sigma_C^R(t) &=  - \int_{-\infty}^{t} {\gamma_{C}^{R}}'(t-x)\left(1-\mu_{C}^{I}\right)\sigma_{E}^{C}(x)\, \d x = \frac{1-\mu_C^I}{T_C} \int_{-\infty}^{t} {\gamma_{C}^{R}}(t-x)\,\sigma_{E}^{C}(x)\, \d x \\
    &= \frac{1-\mu_C^I}{T_C}\, C(t).
    \end{aligned}
    \end{align}
By applying~\eqref{eq:ODE_sigma_EC},  \eqref{eq:ODE_sigma_CI} and~\eqref{eq:ODE_sigma_CR} to compartment $C$ from equation~\eqref{eq:IDEAblKomp}, we get
\begin{align*}
    C'(t) &= \sigma_{E}^{C}(t)-\sigma_{C}^{I}(t)-\sigma_{C}^{R}(t)\\
    &= \frac{1}{T_E} \, E(t) - \frac{\mu_C^I}{T_C} \, C(t) -  \frac{1-\mu_C^I}{T_C}\, C(t)= \frac{1}{T_E} \, E(t) - \frac{1}{T_C} \, C(t).
\end{align*}
The results obtained regarding the compartment $C$ and the associated transitions $\sigma_C^I$ and $\sigma_C^R$ can be equally applied to the remaining compartments. By doing so, we can find the equations for the remaining compartments of the ODE model analogously. In particular, the equations for the compartments $C$ and $I$ are identical for the ODE and the IDE model. This implies that the equation for $S'$ of the models coincide under the given assumptions. This demonstrates that the ODE model is indeed a special case of the IDE model.



\section{Parameters used to demonstrate the convergence order of the IDE model}\label{app:parameters_IDEODE}
In this section, we present the parameters that have been used to demonstrate the convergence order of the IDE model. The parameters are intended to be chosen such that the continuous versions of the ODE and the IDE models are equivalent. Note that these parameters are not intended to model a realistic scenario, but to numerically compare IDE and ODE models. The epidemiological parameters used for the ODE and the IDE model are depicted in~\cref{tab:ide_ode_parameters}. For the IDE model, we use appropriate exponentially distributed stay times, compare also~\ref{app:ide_ode}. This means, that we choose
\begin{align*}\label{eq:ide_ode_gamma}
    \gamma_{z_1}^{z_2}(\tau) = \exp(-\frac{1}{T_{z_1}}\, \tau)
\end{align*} with values $T_{z_1}$ from~\cref{tab:ide_ode_parameters} for appropriate $z_1\in\mathcal{Z}$, see also~\eqref{eq:ide_ode_gamma}.
Furthermore, we \AW{choose} the parameters $\xi_C$, $\xi_I$\AW{, $\rho_C$, and $\rho_I$} \AW{to be constant} so that we have

\begin{align*}
    \xi_C(\tau) = \xi_C \text{\AW{,}} \quad \xi_I(\tau) = \xi_I \quad \AW{\text{and}\quad \rho_C(\tau)=\rho_I(\tau)=\rho}
\end{align*}
for all $\AW{\tau} >0$. Using these parameters, the continuous versions of both models are equivalent. The initial distribution vector 
\begin{align*}
    (S_0, E_0, C_0, I_0, H_0, U_0, R_0, D_0) = (9945, 20, 20, 3, 1, 1, 10, 0)
\end{align*}
is used to initialize the ODE model.
\begin{table}[hbt]
    \centering
    \renewcommand{\arraystretch}{1.1}
    \begin{tabular}{||c|c||c|c||c|c||}
        $T_E$ & $1.4$ & $\mu_C^I$ & $0.5$ & $\phi$ & $1.0$\\
        $T_C$ & $1.2$ & $\mu_I^H$ & $0.5$ & $\rho$& $1.0$\\
        $T_I$ & $0.3$ & $\mu_H^U$ & $0.5$ & $\xi_C$& $1.0$\\
        $T_H$ & $0.3$ & $\mu_U^D$ & $0.5$ & $\xi_I$& $1.0$\\
        $T_U$ & $0.3 $
    \end{tabular}
    \caption{\textbf{Epidemiological parameters used to examine the convergence order.} This choice of parameters is not intended to model a realistic scenario and is applied to both the IDE and the ODE model. }
    \label{tab:ide_ode_parameters}
\end{table}

\section{Parameters to model SARS-CoV-2 dynamics}\label{app:parameters_SARS}
In the following section, we specify the parameters that we use for the IDE and the ODE model to simulate the dynamics of SARS-CoV-2 over a selected time horizon in Germany. We mainly use the results from~\cite{kerr_covasim_2021} and~\cite{kuhn_assessment_2021} and adapt them to our required parameters if necessary.

We need to specify stay time distributions for the IDE model, which are depicted in~\cref{tab:covasim_distributions}. The results are directly taken from~\cite[Table~1]{kerr_covasim_2021} and are repeated here for the sake of completeness and adapted to our notation.

\begin{table}[htb]
    \centering
    \renewcommand{\arraystretch}{1.1}
    \begin{tabular}{c|c}
     Survival function &  Distribution \\
     \hline
     $\gamma_E^C$ & lognormal($4.5$, $1.5$)  \\
     $\gamma_C^I$ & lognormal($1.1$, $0.9$) \\
     $\gamma_C^R$ & lognormal($8.0$, $2.0$)\\
     $\gamma_I^H$ & lognormal($6.6$, $4.9$)\\
     $\gamma_I^R$ & lognormal($8.0$, $2.0$)\\
     $\gamma_H^U$ & lognormal($1.5$, $2.0$)\\
     $\gamma_H^R$ & lognormal($18.1$, $6.3$)\\
     $\gamma_U^D$ & lognormal($10.7$, $4.8$)\\
     $\gamma_U^R$ & lognormal($18.1$, $6.3$)\\
\end{tabular}
\caption{\textbf{Stay time distributions for COVID-19.} The notation is to be understood in such a way that $\gamma_{z_1}^{z_2}$ with $z_1, z_2\in\mathcal{Z}$ follows the corresponding survival function of the lognormal distribution with the parameters (mean, standard deviation). See also~\cite[Table~1]{kerr_covasim_2021}.}\label{tab:covasim_distributions}
\end{table}

The transition probabilities are calculated using~\cite[Table~2]{kuhn_assessment_2021}. 
Since the values there are given per age group, we weight the values with the relative share of the age group in the total population in order to obtain an average value. The weighting is based on the data from~\cite{regionaldatenbank_deutschland_fortschreibung} for $2020$. The value $\mu_U^D$ is calculated using~\cite[Table~2]{kerr_covasim_2021} as \AW{with} this value\AW{, the simulation results better match the reported data}. The source specifies the infection fatality ratio instead of $\mu_U^D$. We can derive our required value from the specified ones by division and weight afterwards again by the relative share of the age group in the total population.
The values resulting from the calculation are shown in~\cref{tab:realistic_prob}.
\begin{table}[hbt]
    \centering
    \renewcommand{\arraystretch}{1.1}
    \begin{tabular}{c|c}
     Transition probability & Value \\
     \hline
    $\mu_C^I$ &  $0.793099$\\
    $\mu_I^H$ & $0.078643$\\
    $\mu_H^U$ & $0.173176$\\
    $\mu_U^D$ &  $0.387803$\\
\end{tabular}
\caption{\textbf{Transition probabilities for COVID-19.} The probabilities are calculated using~\cite[Table~2]{kuhn_assessment_2021} and~\cite[Table~2]{kerr_covasim_2021} and are used for both model types.}\label{tab:realistic_prob}
\end{table}

In the IDE model, we have stay time distributions for the compartments depending on which compartment they are transitioning to in the future. In the ODE model, we do not have this distinction and only have one mean stay time for the respective compartment. Consequently, we set the mean stay times for the compartments of the ODE model by computing a weighted average using the calculated transition probabilities from~\cref{tab:realistic_prob} and the respective mean stay times from the IDE model as given in~\cref{tab:covasim_distributions}. The result can be found in~\cref{tab:realistic_time_ode}.

The transmission probability is assumed to be constant over time and is set to $\AW{\rho_C=\rho_I}=0.0733271$. This value is calculated using~\cite[Table~2]{kuhn_assessment_2021} and again weighted by the share of the age group in the total population of Germany. Additionally, we set $\xi_C=1$ and $\xi_I=0.3$, which are also constant over time. That means, we assume nonsymptomatic individuals to \AW{not} isolate. For symptomatic individuals, we expect them to isolate more frequently.

For simulating change points as described in~\cref{sec:changepoints}, we set the initial contact rate to $\phi(t) = 3.114219$ which results in a constant number of new transmissions in the beginning of the simulation. After two days, we double/half the contact rate.

In the COVID-19 inspired real scenario, we set the initial contact rate $\phi(t) = 7.69129$ for Oct 1, 2020. To simulate the implementation of NPIs, we decrease the contact rate to $\phi(t) = 3.51782$ on Oct 24, 2020. 

\begin{table}[hbt]
    \centering
    \begin{tabular}{c|c}
     Mean stay time & Value \\
     \hline
    $T_E$ & $4.5$\\
    $T_C$ & $2.527617$\\
    $T_I$ & $7.889900$\\
    $T_H$ & $15.225278$\\
    $T_U$ & $15.230258$
\end{tabular}
\caption{\textbf{Average stay time in the compartments for COVID-19.} These mean stay times are used in the ODE model and are calculated using the means of the distributions in~\cref{tab:covasim_distributions} and the probabilities in~\cref{tab:realistic_prob}.}\label{tab:realistic_time_ode}
\end{table}

\end{document}